\documentclass{article}
\usepackage[a4paper,top=3cm,bottom=3cm,left=2cm,right=2cm]{geometry}
\usepackage{amsmath}
\usepackage{amssymb}
\usepackage{amsthm}
\usepackage{color}
\usepackage{multicol}
\usepackage{graphicx}
\usepackage{algpseudocode} 
\usepackage{algorithm} 
\usepackage[]{subfig}
\usepackage{picture}
\usepackage{tikz}
\usepackage{accents}

\newcommand{\z}{\phantom{0}}
\renewcommand{\d}{\mathrm{d}}
\newcommand{\vect}[1]{\boldsymbol{#1}}

\usepackage{adjustbox}

 \usepackage{cases}
 
\newtheorem{ass}{Assumption} 
\newtheorem{lemma}{Lemma}
\newtheorem{prop}{Proposition}
\newtheorem{theorem}{Theorem}

\newtheorem{rmk}{Remark}

\numberwithin{equation}{section}
  
\newcommand{\ds}{\ d s}

\newcommand{\dx}{\ d \vect{x}}
\newcommand{\dFx}{ \ d \vect{\widehat{x}}}

\def\numpatch{\mathcal{N}_{patch}}

\def\B{\color{black}}

\begin{document}

\pagestyle{myheadings}
\markboth{M. Bosy, M. Montardini, G. Sangalli and M. Tani}{}

\title {A domain decomposition method for Isogeometric  multi-patch
  problems with inexact local solvers\thanks{Version of   \today}}

\author{M. Bosy\thanks{Universit\`a di Pavia, Dipartimento di Matematica ``F. Casorati'', 
Via A. Ferrata 1, 27100 Pavia, Italy.} \and M. Montardini $^{\dag}$  \and
  G. Sangalli$^{\dag}$\thanks{IMATI-CNR ``Enrico Magenes'',  Pavia, Italy. \vskip 1mm \noindent Emails: 
{ michal.bosy@unipv.it, monica.montardini01@universitadipavia.it,   giancarlo.sangalli@unipv.it, mattia.tani@imati.cnr.it}}
\and  M. Tani $^{\ddag}$}

%
\maketitle
  
\begin{abstract}
In Isogeometric Analysis, the computational domain is  often
described as  multi-patch, where each patch is given
by a tensor product spline/NURBS parametrization.
In this work we propose  a FETI-like solver where local inexact
solvers exploit the
tensor product structure at the patch level.  To this purpose,  we
extend to the isogeometric framework  the so-called All-Floating
variant of  FETI, that   allows us  to use  the Fast Diagonalization method at
the patch level. We construct then a preconditioner for the whole
system and prove its robustness with respect to the  local mesh-size
$h$ and patch-size $H$ (i.e., we have scalability). Our
numerical tests confirm the theory and also show a favourable
dependence of the computational cost of the method from  the
spline degree $p$. 
 \vskip 1mm
 \noindent
 \textbf{Keywords:}  
  Isogeometric Analysis, domain decomposition, FETI, IETI, preconditioners, Fast Diagonalization.
 \end{abstract}  

\section{Introduction}

Isogeometric Analysis (IgA) was introduced in the seminal paper
\cite{hughes2005isogeometric} as  an extension of   finite
element analysis. The key  idea is to use  the same basis functions that describe the
computational domain, typically B-splines, NURBS or extensions,  also to represent
the unknown solution of the partial differential equations.


In this work, we are concerned with the numerical solution of large
isogeometric compressible  linear elasticity problems in multi-patch
domains, that is,   domains  defined as the union of
several patches, each described by a different tensor product spline/NURBS
parametrization (see \cite{cottrell2009isogeometric}).  Knowing the
advantages that come from  the use of high-degree and 
high-continuity spline approximation (see
for example 
\cite{evans2009n,da2011some,sande2019sharp,takacs2016approximation,bressan2018approximation,MR3811621}),
we are particularly interested in this case. It is also known that the development of linear solvers, both direct and iterative, for
high-degree splines is a challenging task, see \cite{MR2880524,MR3035485}.  

Our starting point is  \cite{MR3572372}, where it has been shown
the potential of the Fast Diagonalization (FD) method to construct
fast solvers for  elliptic isogeometric problems.   The FD method is a
direct solver introduced in \cite{Lynch1964}, that 
 can  be applied to problems with a Sylvester-like structure. In
 general, elliptic isogeometric problems do not possess the required
 Sylvester-like structure, even on a single patch, unless the patch
 parametrization is trivial. However,  \cite{MR3572372}  constructs
 efficient preconditioners (that is, inexact solvers) with the required structure on a single patch. Similarly, here we use FD as
 an inexact and fast solver for problems at the patch level. 

Our approach is based on the   Finite Element Tearing and
Interconnecting   (FETI) idea, that, after its appearance in 
~\cite{MR3618550}, has been widely developed and adopted in finite
element solvers, see \cite{MR2104179}.  IETI, the isogeometric  version of FETI, has been
introduced in \cite{kleiss2012ieti}.   In particular, we develop  in
this paper an All-Floating IETI (in short AF-IETI) method,   the
isogeometric version of the All-Floating FETI introduced
in~\cite{of2009all}, which is in turn  similar to
the so-called total-FETI of~\cite{MR2282408}. With  this variant
of FETI,  both the  global continuity of the solution  and the Dirichlet boundary conditions are
weakly imposed by  Lagrange multipliers. 
The choice of the AF-IETI  formulation is crucial for us since it yields the
Sylvester-like structure that we need to use FD as inexact local solver. 
To allow inexact solvers,  a   saddle point formulation as in
\cite{MR1787293} is also required.

  To show the potential of the proposed inexact AF-IETI, we compare
  numerically its performance to AF-IETI with  the exact local
  solvers. Our results indicate that the inexact approach requires
  orders of magnitude less time than the exact one. Moreover, and
  perhaps even more important, numerical tests indicate that the
  performance of the preconditioner does not deteriorate as
  the degree $p$ is increased. 

Domain decomposition methods represent an active research area  in
isogeometric analysis.  We recall the overlapping Schwarz methods
studied in \cite{bercovier2015overlapping,da2012overlapping},  the  BDDC methods with
related preconditioners studied from \cite{beirao2013bddc}, the
dual-primal approach introduced in  \cite{kleiss2012ieti} and further
studied in \cite{MR3610089,pavarino2016isogeometric}. Domain
decomposition approaches with inexact local solvers have been studied
in    \cite{hofer2017inexact,takacs2018robust}. The case of trimmed domains has ben
recently addressed in \cite{DEPRENTER2019604}

The paper is organized as follows. In Section~\ref{sec:preliminaries} we present  the basics of multi-patch based IgA and in Section~\ref{sec:equations} we introduce the model problem as well as its discrete formulation.    
The AF-IETI method is described in Section~\ref{sec:DB_FETI}, while exact and inexact local solvers  are introduced and analyzed in Section \ref{sec:local-solvers}.  Numerical results are reported in Section \ref{sec:numerics} and, finally, Section \ref{sec:con} contains some conclusions and future directions of research.

 
\section{Preliminaries  } \label{sec:preliminaries}
   \subsection{B-splines}
   \label{sec:splines}
Given two integers $m, p > 0$,   we introduce a \emph{knot vector} $\Xi:=\{0=\xi_1\leq \dotsc\leq \xi_{m+p+1}=1\}$ in the interval $[0,1]$, where $m$ and $p$ are, respectively, the number of basis functions that will be built from the knot vector and their polynomial degree.
We consider \emph{open} knot vectors, i.e. we set $\xi_1=\dotsc=\xi_{p+1}=0$ and $\xi_{m+1}=\dotsc=\xi_{m+p+1}=1$. 
Following Cox-de Boor recursion formulas \cite{MR1900298},  univariate B-splines are piecewise polynomials   defined for $i=1,\dots,m$ as follows:\par
for $p = 0$ 
\begin{eqnarray*}
\widehat{b}_{ i,0}(\eta)= \begin{cases}1 & { \textrm{if }} \xi_{i}\leq \eta<\xi_{i+1},\\
0 & \textrm{otherwise},
\end{cases}
\end{eqnarray*}
\indent for $p \geq 1$
\begin{eqnarray*}
\widehat{b}_{ i,p}(\eta)=\! \begin{cases}\dfrac{\eta-\xi_{i}}{\xi_{i+p}-\xi_{i}}\widehat{b}_{ i, p-1}(\eta) +\dfrac{\xi_{i+p+1}-\eta}{\xi_{i+p+1}-\xi_{i+1}}\widehat{b}_{ i+1, p-1}(\eta) & { \textrm{if }} \xi_{i}\leq \eta<\xi_{i+p+1}, \\[8pt]
0 & \textrm{otherwise,}
\end{cases}
\end{eqnarray*}
 where we adopted the convention $\frac{0}{0}=0$. The multiplicity of the internal knots influences the smoothness of the B-splines (see \cite{cottrell2009isogeometric}).
The corresponding \emph{ univariate spline space} is defined as
\begin{equation*}
\widehat{\mathcal{S}}_p: = \mathrm{span}\{\widehat{b}_{i,p}\}_{i = 1}^m. 
\end{equation*}
 We also introduce the mesh-size $h:=\max\{\xi_{i+1}-\xi_i \ | \ i=1,\dots, m+p\}$.
  We remark that the first and last basis function are nodal at the endpoint of the unit interval, i.e. $ \widehat{b}_{1,p}\left(0\right) = \widehat{b}_{m,p}\left(1\right) = 1$.


We consider multivariate B-splines as tensor products of 
univariate ones. In particular, for $d$-dimensional problems, given $2d$ integers $m_l,p_l>0$ for $l=1,\dots,d,$ we introduce $d$ univariate knot vectors $\Xi_l:=\{\xi_{l,1},\dots,\xi_{l,m_{l}+p_l+1}\}$   for
$l=1 ,.., d$ and the corresponding mesh-sizes denoted with $h_l$  for
$l=1 ,.., d$.   For simplicity we suppose that the degree of the B-splines is the same in all parametric direction, i.e. we set  $p_1=\dots=p_d=:p$, but the general case is similar.   For a multi-index $\vect{i} = (i_1, \dotsc, i_d)$, we define the multivariate B-spline as follows
\begin{equation*}
\widehat{B}_{\vect{i},  {p}}(\vect{\eta}) : = \widehat{b}_{i_1, p }(\eta_1) \dotsc \widehat{b}_{i_d, p }(\eta_d)
\end{equation*}
 where $\vect{\eta} = (\eta_1, \dotsc, \eta_d)$. Hence, the \emph{multivariate spline space} on the parametric domain $\widehat{\Omega}:=\left[0,1\right]^d$ is defined as
\begin{equation*}
 \boldsymbol{\widehat{\mathcal{S}}}_{{p}}:= \underbrace{ \widehat{\mathcal{S}}_{p} \times \dots \times \widehat{\mathcal{S}}_{p}}_{d}=\mathrm{span}\{\widehat{B}_{\vect{i}, {p}} \ | \ i_l = 1,\dotsc, m_{l}; l=1,\dotsc, d \}.
 \end{equation*}
 We also introduce the global mesh-size, defined as $h:=\max\{h_l\ | \ l=1,\dots, d\}$.
\subsection{Multi-patch domains}
\label{sec:multipatch}
  Let $\Omega\subset\mathbb{R}^d$    be  the union of $\numpatch$     isogeometric patches, i.e $\overline{\Omega} = \bigcup_{k = 1}^{\numpatch} \overline{\Omega}^{(k)}$ and $\Omega^{(j)} \cap \Omega^{(k)} = \emptyset$ for $j \neq k$.
  In the present approach, these patches coincide with the non-overlapping subdomains whose prescription is the starting point of every FETI method.   Let $H^{(k)}$ be the diameter of $\Omega^{(k)}$ for $k=1,\dots,\numpatch$ and   $H:=\max\{H^{(k)} \ | \  k=1,\dots,\numpatch\}$.   
  For each   patch, given $d+1$ integers $m_1^{(k)},\dots,m_d^{(k)},p^{(k)}>0$, we introduce open knot vectors $\Xi_l^{(k)}:=\{0=\xi^{(k)}_1\leq \dotsc\leq \xi^{(k)}_{m^{(k)}_l+p^{(k)}+1}=1\}$ for $l=1,\dots,d$ and the local mesh-size $h^{(k)}$.  For simplicity, we suppose that the degree of the B-splines is the same in each patch, i.e. we set $p^{(1)}=\dots=p^{(\numpatch)}=:p$, even if the general case is similar. Let also $h:=\max\{h^{(k)}\ | \ k=1,\dots, \numpatch\}$.
We will make the following assumption on the quasi-regularity of the meshes and on the diameter of the patches.
 \begin{ass}
\label{ass:quasi-uniformity}
 There exists   $\alpha\in(0,1] $,   independent
of $h^{(k)}$ and $H^{(k)}$ for $k=1,\dots,\numpatch$, such that each non-empty knot span $( \xi^{(k)}_{l,i} ,
\xi^{(k)}_{l,i+1})$ fulfils $ \alpha h^{(k)}  \leq \xi_{l,i+1}^{(k)} - \xi_{l,i}^{(k)} \leq
h^{(k)}$  for $i=1,\dots m^{(k)}_l+p$, for $l=1,\dots  d$ and for $k=1,\dots,\numpatch$   and  each   $H^{(k)}$ fulfils $ \alpha H  \leq H^{(k)} \leq
H$ for $k=1,\dots,\numpatch$.
\end{ass}  
  
  We denote the multivariate spline-space associated as $\boldsymbol{\widehat{\mathcal{S}}}^{(k)}_{ {p}}:=\mathrm{span}\{ \widehat{B}^{(k)}_{\vect{i}, {p}} \ | \ i_l = 1,\dotsc, m^{(k)}_{l}; l=1,\dotsc, d  \}$.
By introducing a colexicographical reordering of the basis functions, we have
\begin{equation*} 
\boldsymbol{\widehat{\mathcal{S}}}^{(k)}_{p}:=\mathrm{span}\left\{ \widehat{B}^{(k)}_{ {i},p} \ | \ i = 1,\dots,n^{(k)} \right\}, 
\end{equation*}
  where $n^{(k)}:=\mathrm{dim}(\boldsymbol{\widehat{\mathcal{S}}}^{(k)}_{p})$.  
Each $\Omega^{(k)}$ is represented by a  non-singular  spline parametrization $\mathcal{F}^{(k)}\in \left[ \vect{\widehat{\mathcal{S}}}^{(k)}_{p}\right]^d$,   i.e. $\Omega^{(k)} = \mathcal{F}^{(k)}(\widehat{\Omega})$ and the Jacobian matrix $J_{\mathcal{F}^{(k)}}$ is invertible everywhere.  
According to the isoparametric concept, the \emph{isogeometric space} on each patch is defined as 
\begin{equation}
V^{(k)}_{h}:= \mathrm{span}\left\{{B}_{ {i}, p}^{(k)}:=\widehat{B}^{(k)}_{ {i}, p} \circ \left(\mathcal{F}^{(k)}\right)^{-1} \  \bigg| \ i = 1,\dotsc, n^{(k)}  \right\},
\label{eq:univDBFETI-space}
\end{equation} 
while the \emph{isogeometric space} over $\Omega$ is defined as
\begin{equation*}
V_h:=\Pi_{k=1}^{\numpatch} V_h^{(k)} = \left\{ v\in L^2(\Omega)\ \bigg| \ v_{|_{\Omega^{(k)}}}\in V_h^{(k)}, \ k=1,\dots,\numpatch\right\}.
\end{equation*} 
Note that functions in $V_h$ are not necessarily continuous.

Through this paper, we consider both \emph{conforming} and \emph{non-conforming} meshes at the patch interfaces. 
In the first case, for all $j$ and $k$ s.t. $\partial\Omega^{(k)}\cap\partial\Omega^{(j)}\neq\emptyset$ and this intersection is not a point, $\Omega^{(k)}$ and $\Omega^{(j)}$ are fully matching (see \cite{hughes2005isogeometric}).
In the second case, we allow that on a given interface the knot vector of one patch can be more refined than the knot vector of an   adjacent patch   (see Figure~\ref{fig:poisson_3d} in Section \ref{sec:numerics}). With this assumption, we have that splines of the coarser side can be written as a linear combination  of splines of the finer one.


\section{Model problem and its discretization}
\label{sec:equations}

Let $\Omega\subset\mathbb{R}^d$ be a computational domain  described by a multi-patch spline parametrization, as in Section \ref{sec:multipatch}, and let $\partial \Omega$ denote its  boundary.  Suppose that  $\partial \Omega=\partial\Omega_D\cup\partial\Omega_N$ with $\partial\Omega_D\cap \partial\Omega_N=\emptyset$, where $\partial\Omega_D$ has positive measure.   
Let $\vect{f} \in  [{L}^2(\Omega)]^d $,  $ \ \vect{g} \in   [{L}^2(\partial \Omega_N)]^d$ and $H^1_D(\Omega):=\{v \in H^1(\Omega)\ \text{s.t.} \ v=0 \text{ on } \partial\Omega_D\}$.  Then, the  variational formulation of the compressible linear elasticity   problem    we consider reads:
 \begin{center}
\textit{Find  $\vect{u}\in [H^1_D(\Omega)]^d$    s.t. for all   $\vect{v}\in [H^1_D(\Omega)]^d$ }\end{center}
\begin{equation*}
a(\vect{u},\vect{v})= \langle \vect{F}, \vect{v} \rangle ,
\end{equation*}
where we define
\begin{align}
\label{eq:bil_lin}
a(\vect{u},\vect{v}) :=  2 \mu \int_{\Omega} \varepsilon(\vect{u}) : \varepsilon(\vect{v}) \dx + \lambda \int_{\Omega} \left( \nabla \cdot \vect{u} \right) \left(\nabla \cdot \vect{v}\right) \dx    , \quad
\langle \vect{F}, \vect{v} \rangle :=  \int_{\Omega} \vect{f} \cdot \vect{v} \dx + \int_{\partial \Omega_N} \vect{g}\cdot \boldsymbol{v}    \ds 
\end{align}
and where we used the notation     $ {\varepsilon}(\boldsymbol{v}) := \frac{1}{2} \left(\nabla \boldsymbol{v}+ (\nabla \boldsymbol{v})^T\right)$ for the symmetric gradient,   while   $\lambda$ and $\mu$ denote the material Lam\'{e} coefficients.

The corresponding discrete problem we want to solve is then
\begin{center}
\textit{Find $\vect{u}_h \in [ V_{h}\cap  H^1_D(\Omega) ]^d$ s.t. for all $\vect{v}_h \in [ V_{h}\cap  H^1_D(\Omega)]^d$}                                                                                                                                                                        \end{center}
\begin{equation}
\label{eq:discrete_elasticity} a(\vect{u}_h, \vect{v}_h)  = \langle \vect{F}, \vect{v}_h \rangle.
\end{equation}


\section{All-Floating IETI method} 
\label{sec:DB_FETI}

In this section, we present  the All-Floating IETI (AF-IETI) method, an  extension to IgA  of the AF-FETI method, introduced in \cite{pechstein2012finite}, (see also \cite{pechPHD}).  In this formulation, the FETI interface includes   the whole boundary $\partial \Omega$ without distinction between Dirichlet and Neumann boundary.  

Let $\Omega$ be the computational domain described as the  union of $\numpatch$  isogeometric patches $\Omega^{(k)}$ for $k = 1,\ldots,\numpatch$,  as detailed  in  Section \ref{sec:multipatch}.  
%
We  work with the  computational spaces 
\begin{align}
\label{eq:t-feti_space}
\mathcal{V}_h^{(k)}  := 
\left[V_h^{(k)}\right]^d \quad \text{and} \quad\mathcal{V}_h := \Pi_{k=1}^{\numpatch} \mathcal{V}_h^{(k)},
\end{align}
where $V_h^{(k)}$ is defined  in \eqref{eq:univDBFETI-space}.
  Note that the space   $\mathcal{V}_h$  above does not include any boundary condition or continuity condition   across the patch boundaries,   unlike  the space $[V_h\cap H^1_D(\Omega)]^d$ used in problem \eqref{eq:discrete_elasticity}.  
Let 
\begin{equation*}
\mathcal{N}_{dof}^{(k)} :=\mathrm{dim}\left(\mathcal{V}_h^{(k)}\right)=\left[ n^{(k)}\right]^d  \quad \text{ and }\quad \mathcal{N}_{dof}:=\mathrm{dim}\left(\mathcal{V}_h\right) = \sum_{k=1}^{\numpatch}\mathcal{N}_{dof}^{(k)}
\end{equation*}
denote  the dimension of the $k-$th local space  for  $k=1,\ldots, \numpatch$, and the dimension of the whole space, respectively. Each function $ \vect{v}^{(k)}_h\in\mathcal{V}_h^{(k)}$ is uniquely represented by  a coordinate  vector $\mathbf{v}^{(k)}\in\mathbb{R}^{\mathcal{N}^{(k)}_{dof}}$. Similarly, each function $\vect{v}_h\in\mathcal{V}_h $ is associated to a coordinate vector $\mathbf{v}\in\mathbb{R}^{\mathcal{N}_{dof}}$ of the form 
  $$
  \mathbf{v}:=\begin{bmatrix}
  \mathbf{v}^{(1)}\\
  \vdots\\
  \mathbf{v}^{(\numpatch)}
  \end{bmatrix} .
  $$ 
 
We  assemble the local stiffness matrices $\mathbf{A}^{(k)}$ and local right-hand-side vectors $\mathbf{f}^{(k)}$ by integrating the appropriate expressions over individual patches $\Omega^{(k)}$ for $k=1,\dots,\numpatch$.
The matrices $ \mathbf{A}^{(k)}$ for $k=1,\dots,\numpatch $ are symmetric, positive semidefinite and singular. We then define the block-diagonal matrix $\mathbf{A}$ and the load vector $\mathbf{f}$ as 
  \begin{align}  
\label{eq:block_A}
 \mathbf{A} := 
 	\begin{bmatrix}
 		 \mathbf{A}^{(1)} & &    \\
 		 & \ddots & \\
 		  & & \mathbf{A}^{(\numpatch)}
 	\end{bmatrix}	\in \mathbb{R}^{\mathcal{N}_{dof} \times \mathcal{N}_{dof}}\quad  \text{and}\quad 
 \mathbf{f} := 
 \begin{bmatrix}
 		 \mathbf{f}^{(1)} \\
 		 \vdots \\
 		 \mathbf{f}^{(\numpatch)}
 	\end{bmatrix}  \in \mathbb{R}^{\mathcal{N}_{dof}} 
 \end{align}  
where each local matrix $ \mathbf{A}^{(k)}$ has a natural block structure, that follows from the vectorial nature of the space $\mathcal{V}_h^{(k)}$ in \eqref{eq:t-feti_space} and of the bilinear form $a(\cdot, \cdot)$ in \eqref{eq:bil_lin}:
 \begin{equation}
\label{eq:le_local_matrix}
 \mathbf{A}^{(k)} := 
 	\begin{bmatrix}
 		 \mathbf{A}^{(k)}_{1,1} &\dots & \mathbf{A}^{(k)}_{1,d} \\
 		 \vdots & \ddots & \vdots \\
 		\mathbf{A}^{(k)}_{d,1} & \dots & \mathbf{A}^{(k)}_{d,d}
 	\end{bmatrix}.
 \end{equation}
  The interface $\Gamma$ is   defined as the union of the local interfaces, that in our method are the whole local boundaries $\partial \Omega^{(k)}$, as 
\begin{equation*}
\Gamma :=  \bigcup_{k = 1}^{\numpatch}  \partial \Omega^{(k)}.
\end{equation*}

Note that, thanks to the use of open knot vectors, we can identify the local degrees-of-freedom associated to basis functions   that have non-zero support on the local interface $\partial\Omega^{(k)}$ and the remaining degrees-of-freedom,  that are associated to the interior functions.  
Thus, we also  assume a partition of each local matrix $\mathbf{A}^{(k)}$ as
\begin{equation}
\label{eq:A_partition}
\mathbf{A}^{(k)} := 
 	\begin{bmatrix}
 		 \mathbf{A}^{(k)}_{II} &   \mathbf{A}^{(k)}_{I \Gamma } \\
 		\mathbf{A}^{(k)}_{\Gamma I} &   \mathbf{A}^{(k)}_{\Gamma  \Gamma }
 	\end{bmatrix},
\end{equation}
where the subscript $\Gamma $ refers to the interface degrees-of-freedom that belongs to the $k$-th patch, while $I$ indicates the remaining local interior degrees-of-freedom. 
Let also $\mathcal{N}_{\Gamma}^{(k)}$ be the dimension of $\mathbf{A}^{(k)}_{\Gamma  \Gamma }$ for $k=1,\dots,\numpatch$, i.e. $\mathbf{A}^{(k)}_{\Gamma  \Gamma }\in\mathbb{R}^{\mathcal{N}_{\Gamma}^{(k)}\times\mathcal{N}_{\Gamma}^{(k)}}$ for $k=1,\dots,\numpatch$. 

 As the functions in   $\mathcal{V}_h$ are in general discontinuous across the patch boundaries and do not have prescribed Dirichlet boundary conditions, we need to impose these constraints separately. To this end, we introduce 
 the sparse matrix  $\mathbf{B}\in\mathbb{R}^{\mathcal{N}_c\times \mathcal{N}_{dof}}$, where   $\mathcal{N}_c$ denotes the number of  conditions to impose.
In particular,  $\mathbf{B}\mathbf{v}=\mathbf{0}$ when the  function $\vect{v}_h\in\mathcal{V}_h$ corresponding to the vector $\mathbf{v}\in\mathbb{R}^{\mathcal{N}_{dof}}$ satisfies   homogeneous Dirichlet boundary conditions and it is continuous across the patches. 
If Dirichlet boundary conditions are non-homogeneous, then also the right-hand side has to be modified (see \cite{MR2282408} for further details).  Without loss of generality, we assume that $\mathrm{ker}(\mathbf{B}^T)=\vect{0}$, i.e. there are no redundant constraints and $\mathrm{range}(\mathbf{B})=\mathbb{R}^{\mathcal{N}_c}$.

By introducing a vector of Lagrange multipliers $\vect{\lambda}\in \mathbb{R}^{\mathcal{N}_c}$,  problem 
\eqref{eq:discrete_elasticity} can be reformulated as follows:

 \begin{center}
 \textit{Find $(\mathbf{u}, \vect{\lambda}) \in  \mathbb{R}^{\mathcal{N}_{dof}}\times \mathbb{R}^{\mathcal{N}_c}$ such that}
 \begin{equation}
 \label{eq:db_feti_system}
 \begin{bmatrix}
 		 \mathbf{A} & \mathbf{B}^T \\
 		 \mathbf{B} & \mathbf{0}
 	\end{bmatrix}
	\begin{bmatrix}
 		 \mathbf{u} \\
 		 \vect{\lambda}
 	\end{bmatrix}
	=
	\begin{bmatrix}
 		\mathbf{f}\\
 		\mathbf{0 }
 	\end{bmatrix}.
 \end{equation}
 \end{center} 
  The difficulty in preconditioning \eqref{eq:db_feti_system} is that $\mathbf{A}$ is singular.
Following \cite{MR1787293}, we  introduce  the matrix
 
%
%
%

\begin{equation*}
\mathbf{R} :=\begin{bmatrix}
 \mathbf{R}^{(1)}& &   \\
 		 & \ddots &  \\
 		  & &\mathbf{R}^{(\numpatch)} 
\end{bmatrix},
\end{equation*}
where the matrices $\mathbf{R}^{(k)}$ are defined such that $\mathrm{range}(\mathbf{R}^{(k)})=\mathrm{ker}(\mathbf{A}^{(k)})$ for $k = 1,\ldots\numpatch$.
Thus, we also have that  $\mathrm{range}(\mathbf{R})=\mathrm{ker}( \mathbf{A})$.	
In our case, $\mathbf{R}^{(k)}$ represents 
the space of rigid body motions on the patch $\Omega^{(k)}$. 
For three-dimensional problems, the case addressed in the numerical experiments of this paper,  this space is spanned by three translations and three   infinitesimal  rotations 
\begin{align*}
\vect{r}_1:=\begin{bmatrix}
1\\0\\0
\end{bmatrix}
&,&
\vect{r}_2:=\begin{bmatrix}
0\\1\\0
\end{bmatrix},\quad
\vect{r}_3:=\begin{bmatrix}
0\\0\\1
\end{bmatrix},\quad
\vect{r}_4:=\begin{bmatrix}
0\\ \vect{x}_3\\-\vect{x}_2
\end{bmatrix}
&,&
\vect{r}_5:=\begin{bmatrix}
\vect{x}_3 \\0\\ -\vect{x}_1
\end{bmatrix},\quad
\vect{r}_6:=\begin{bmatrix}
\vect{x}_2\\ -\vect{x}_1\\0
\end{bmatrix}.
\end{align*}
 Note that $\mathrm{ker}(\mathbf{A})\cap\mathrm{ker}(\mathbf{B})=\vect{0}$, hence system \eqref{eq:db_feti_system} is uniquely solvable.

Multiplying the first block of equations of~\eqref{eq:db_feti_system} by $\mathbf{R}^T$ and using that $\mathbf{R}^T \mathbf{A} = \mathbf{0}$, we get
\begin{equation*}
\mathbf{R}^T \mathbf{B}^T \vect{\lambda} = \mathbf{R}^T \mathbf{f}.
\end{equation*}
Let  $\mathbf{G} := \mathbf{B} \mathbf{R}$ and  decompose $\vect{\lambda} = \vect{\lambda}_0 + \vect{\chi}$ such that $\vect{\lambda}_0$ satisfies
\begin{equation*}
\mathbf{G}^T \vect{\lambda}_0 = \mathbf{R}^T \mathbf{f},
\end{equation*}
and $\vect{\chi} \in \mbox{ker}(\mathbf{G}^T)$.
We introduce the orthogonal projection onto $\mbox{ker}(\mathbf{G}^T)$
\begin{equation}
\label{eq:proj_chi}
\mathcal{P}_{ {\chi}} := \mathbf{I}_{\mathcal{N}_c} - \mathbf{G} \left(\mathbf{G}^T \mathbf{G}\right)^{-1} \mathbf{G}^T \in \mathbb{R}^{ \mathcal{N}_c\times \mathcal{N}_c}, 
\end{equation}
where $\mathbf{I}_{n}$ denotes the identity matrix of dimension $n\times n$.
Finally, the problem we want to solve is the following:
 \begin{center}
 \textit{Find $(\mathbf{w}, \vect{\chi}) \in \mbox{range}(\mathbf{A}) \times \mbox{ker}(\mathbf{G}^T)$ such that} 
 \begin{equation}
 \label{eq:db_feti_alternative_system}
 \mathcal{A} 
 \begin{bmatrix}
 		 \mathbf{w} \\
 		 \vect{\chi}
 	\end{bmatrix} :=
 \begin{bmatrix}
 		 \mathbf{A} & (\mathcal{P}_{ {\chi}} \mathbf{B})^T \\
 		 \mathcal{P}_{ {\chi}} \mathbf{B} & \mathbf{0}
 	\end{bmatrix}
	\begin{bmatrix}
 		 \mathbf{w} \\
 		 \vect{\chi}
 	\end{bmatrix}
	=
	\begin{bmatrix}
 		 \mathbf{f}- \mathbf{B}^T \vect{\lambda}_0 \\
 		 \mathbf{0}
 	\end{bmatrix} .
 \end{equation}
 \end{center}
We observe that $\mathcal{A}$ is an isomorphism of $\mbox{range}( \mathbf{A}) \times \mbox{ker}(\mathbf{G}^T)$ into itself. Indeed, $\mathbf{A}$ is an isomorphism on $\mbox{range}(\mathbf{A})$, and $(\mathcal{P}_{\chi} \mathbf{B})^T \mathbf{v} = \mathbf{B}^T \mathcal{P}_{ {\chi}} \mathbf{v} = \mathbf{B}^T \mathbf{v}$ for every $\mathbf{v} \in \mbox{ker}(\mathbf{G}^T)$. Thus, system~\eqref{eq:db_feti_alternative_system} is uniquely solvable. Note that,   since $\mathcal{P}_{ {\chi}} \vect{\chi} = \vect{\chi}$ and $\mathcal{P}_{ {\chi}} = \mathcal{P}_{ {\chi}}^T$, the first equation of~\eqref{eq:db_feti_alternative_system} is equivalent to  the first equation of ~\eqref{eq:db_feti_system}. However, from the second equation of ~\eqref{eq:db_feti_alternative_system} we see that the solution $\mathbf{w}$ satisfies $\mathcal{P}_{ {\chi}} \mathbf{B w} = \mathbf{0}  $ and not $\mathbf{B w} = \mathbf{0} $.

Thus, given    the  solution $\begin{bmatrix}
\mathbf{w} \\ \boldsymbol{\chi} 
\end{bmatrix}$  of~\eqref{eq:db_feti_alternative_system},   by a straightforword calculation, we can see that the solution $\begin{bmatrix}
\mathbf{u}\\ \boldsymbol{\lambda}
\end{bmatrix}$   of~\eqref{eq:db_feti_system}   is
\begin{align*}
\mathbf{u}:= \left( \mathbf{I}_{\mathcal{N}_{dof}} - \mathbf{R} \left(\mathbf{G}^T \mathbf{G}\right)^{-1} \mathbf{G}^T \mathbf{B} \right) \mathbf{w}
 &, &
 \vect{\lambda} := \vect{\lambda}_0 + \vect{\chi}.
\end{align*}


  Our   solver for the linear system \eqref{eq:db_feti_alternative_system} is  MINRES   preconditioned by a block-diagonal matrix, whose construction and required properties are discussed below.

Using the splitting of the local matrices into boundary and internal degrees-of-freedom \eqref{eq:A_partition}, we   introduce the local Schur complement 
  matrices $\mathbf{S}^{(k)}\in\mathbb{R}^{\mathcal{N}_{\Gamma}^{(k)}\times \mathcal{N}_{\Gamma}^{(k)}}$ for $k=1,\dots,\numpatch$, defined as
 \begin{equation*}
\mathbf{S}^{(k)} :=  \mathbf{{A}}^{(k)}_{\Gamma \Gamma }- \mathbf{{A}}^{(k)} _{\Gamma I}  \left( \mathbf{{A}}^{(k)} _{I I}\right)^{-1} \mathbf{{A}}^{(k)}_{I \Gamma } \quad \text{for } k = 1, \cdots, \numpatch.
\end{equation*} 
  Let also $\mathbf{M}^{(k)}  \in \mathbb{R}^{\mathcal{N}_{dof}^{(k)} \times \mathcal{N}_{dof}^{(k)} }$ for $k=1,\dots,\numpatch$ be the local mass matrices, that are block-diagonal matrices, split component-wise as
 \begin{equation}
 \label{eq:le_local_MS}
 \mathbf{M}^{(k)} := 
 	\begin{bmatrix}
 		 \mathbf{M}^{(k)}_1 & &   \\
 		 & \ddots & \\
 		  & & \mathbf{M}^{(k)}_d
 	\end{bmatrix}.
 \end{equation} 
 We then can build the global Schur complement and mass matrices, as 
$$
\mathbf{S}:=\begin{bmatrix}
\mathbf{S}^{(1)} & &  \\
& \ddots & \\
  & & \mathbf{S}^{(\numpatch)}
\end{bmatrix} \in \mathbb{R}^{\mathcal{N}_{\Gamma}\times\mathcal{N}_{\Gamma}}\quad  \quad \mathbf{M}:=\begin{bmatrix}
\mathbf{M}^{(1)} & &  \\
& \ddots & \\
  & & \mathbf{M}^{(\numpatch)}
\end{bmatrix} \in \mathbb{R}^{\mathcal{N}_{dof}\times\mathcal{N}_{dof}},
$$ where $\mathcal{N}_{\Gamma}:=\sum_{k=1}^{\mathcal{N}_{patch}}\mathcal{N}_{\Gamma}^{(k)}$ represents the number of degrees-of-freedom
associated to the interface $\Gamma$. Note that  $\mathbf{S}$ is the
Schur complement   of $\mathbf{A}$ that is obtained by eliminating the
interior degrees-of-freedom of each patch.   We also introduce the scaling
diagonal matrix
$$\mathbf{D}_H := \begin{bmatrix}  (H^{(1)})^{-2}\ \mathbf{I}_{\mathcal{N}_{dof}^{(1)}} & & \\
 & \ddots & \\
 & &  (H^{(\numpatch)})^{-2}\ \mathbf{I}_{\mathcal{N}_{dof}^{(\numpatch)}} \end{bmatrix},$$ 
and the matrix $\mathbf{B}_{\Gamma}\in\mathbb{R}^{\mathcal{N}_{c}\times
  \mathcal{N}_{\Gamma}}$ as the restriction of the constraint matrix
$\mathbf{B}$ to the interface degrees-of-freedom. 
Finally, following~\cite{MR1787293},   we define  the block-diagonal  preconditioner 
\begin{equation}
\label{eq:all_preconditioner}
 \mathcal{B}^{-1} :=
 \begin{bmatrix}
 		\mathcal{P}_{{u}} \mathbf{P}_{A}^{-1}   \mathcal{P}_{{u}}^T &    \\
 		  & \mathcal{P}_{{\chi}}   \mathbf{P}_{{S}}    \mathcal{P}_{{\chi}}^T
 	\end{bmatrix}, 
\end{equation}
for the system~\eqref{eq:db_feti_alternative_system}, where
$$
\mathcal{P}_{{u}} :=  \mathbf{I}_{\mathcal{N}_{dof}} - \mathbf{R} \left(\mathbf{R}^T \mathbf{R}\right)^{-1} \mathbf{R}^T 
$$
is the orthogonal projection onto $\mbox{range}(\mathbf{A})$ and
$\mathcal{P}_{\vect{\chi}}$ is the orthogonal projector onto
$\mbox{ker}(\mathbf{G}^T)$ defined in \eqref{eq:proj_chi}. The
following general result provides the bound on the condition number for
the resulting preconditioned system.

\begin{theorem} 
\label{tw:cond_num}
   Under the assumptions above, given two symmetric  matrices $\mathbf{P}_{A} \in
   \mathbb{R}^{\mathcal{N}_{dof} \times \mathcal{N}_{dof}}$ and
   $\mathbf{P}_{S} \in \mathbb{R}^{\mathcal{N}_{\Gamma} \times
     \mathcal{N}_{\Gamma}}$, if there exist  positive
   constants  $a_0,a_1,s_0$ and $s_1$ independent of $h$ and $H$ such that 
\begin{equation} 
a_0 \, \mathbf{u}^T(\mathbf{A} + \mathbf{D}_H \mathbf{M})\mathbf{u} \; \leq   \; \mathbf{u}^T\mathbf{P}_{{A}}\mathbf{u}\;   \leq \; a_1 \, \mathbf{u}^T(\mathbf{A} + \mathbf{D}_H \mathbf{M})\mathbf{u}\quad   \forall \  \mathbf{u}\in \mathrm{range}(\mathbf{A}), \label{eq:bound_A} 
\end{equation}
\begin{equation}\\
s_0 \, \vect{\lambda}^T\mathbf{B}_{\Gamma}\mathbf{S}\mathbf{B}_{\Gamma}^T \vect{\lambda} \; \leq   \; \vect{\lambda}^T \mathbf{P}_{S} \vect{\lambda} \;   \leq \; s_1 \, \vect{\lambda}^T\mathbf{B}_{\Gamma}\mathbf{S}\mathbf{B}_{\Gamma}^T \vect{\lambda} \quad   \forall \  \vect{\lambda}\in \mathrm{ker}(\mathbf{G}^T),\label{eq:bound_S}
\end{equation}
 then the condition number of the system
 ~\eqref{eq:db_feti_alternative_system} preconditioned by
 \eqref{eq:all_preconditioner} fulfils
\begin{equation}
\label{eq:cond_number}
\kappa (\mathcal{B}^{-1} \mathcal{A}) \leq C \left(1+ \log{\frac{H}{h}}\right)^2,
\end{equation}
where $C > 0$ is a constant  independent of $h$ and $H$.
\end{theorem}
\begin{proof}
  The proof is analogous to the one of \cite[Theorem 15]{MR1787293},
  which follows from \cite[Lemma 13]{MR1787293} and  \cite[Lemma
  14]{MR1787293} that can be straightforwardly extended to our
  framework.  We remark that the proof of \cite[Lemma 13]{MR1787293}
  uses inverse inequalities and, for that reason, the constant $C$ in
  \eqref{eq:cond_number} may depend on the spline degree $p$. 
\end{proof}
 
Note in \eqref{eq:all_preconditioner} the different roles of the matrix $\mathbf{P}_A$, that needs to be inverted, and of the matrix $\mathbf{P}_S$, that is just multiplied.
 We further observe that $\mathcal{B}^{-1}$ is an isomorphism of $\mbox{range}(  \mathbf{A}) \times \mbox{ker}( \mathbf{G}^T)$ into itself. Hence the preconditioned problem
 \begin{center}
 \textit{Find $(\mathbf{w}, \vect{\chi}) \in\mbox{range}(  \mathbf{A}) \times \mbox{ker}( \mathbf{G}^T)$ such that:}
 \begin{equation*}
 \mathcal{B}^{-1} \mathcal{A} 
 \begin{bmatrix}
 		 \mathbf{w} \\
 		 \vect{\chi}
 	\end{bmatrix}
	= \mathcal{B}^{-1}
	\begin{bmatrix}
 		\mathbf{f}- \mathbf{B}^T \vect{\lambda}_0 \\
 		 \mathbf{0}
 	\end{bmatrix} 
 \end{equation*}
 \end{center}
  is uniquely solvable and equivalent to \eqref{eq:db_feti_alternative_system}. 
Finally, the AF-IETI method is the MINRES method applied to the preconditioned system above.
\section{Solving the local problems}
 \label{sec:local-solvers} 
This section deals with the definition of $\mathbf{P}_A$ and
 $\mathbf{P}_S$. These operators are selected block-diagonal, where each
 block corresponds to a patch. In particular, both the application of
 $\mathbf{P}_A^{-1} $ and the application of 
 $\mathbf{P}_S$    correspond to the solution of patch-wise elliptic
 problems.    We discuss two possible choices: the first  involves
 exact solvers   for  $\mathbf{A} + \mathbf{D}_H\mathbf{M}$ and
 $\mathbf{B}_{\Gamma}\mathbf{S}\mathbf{B}^T_{\Gamma}$,   while  the
 second represents an inexact version that makes the application of
 the whole  preconditioner $\mathcal{B}^{-1}$ more efficient.
 
\subsection{Exact local solvers}
\label{sec:exact} 
From \eqref{eq:bound_A} and \eqref{eq:bound_S}, we infer that the ideal choice for $\mathbf{P}_{A}$ and $\mathbf{P}_{S}$, respectively,  is
\begin{align}  
\mathbf{P}_{A}^{\mathrm{ex}}& =\mathbf{A} +
                            \mathbf{D}_H\mathbf{M}, \label{eq:exact_AP}\\   
 \mathbf{P}^{\mathrm{ex}}_{S}&=\mathbf{B}_{\Gamma}\mathbf{S}\mathbf{B}^T_{\Gamma}.\label{eq:exact_SP}  
\end{align}   
The matrix \eqref{eq:exact_SP}  is known in the FETI community as    Dirichlet preconditioner (see \cite{farhat1994optimal,mandel1996convergence}), as its application involves the solution of Dirichlet problems.

\subsection{Inexact local solvers}
\label{sec:inexact}
 Before introducing the inexact local solvers, we recall the definition of the Kroneker product.  

Let  $\mathbf{C}\in \mathbb{R}^{n_c\times n_c}$ and $\mathbf{D}\in \mathbb{R}^{n_d \times n_d}$ be square matrices and let the entries of the matrix $\mathbf{C}$ be denoted with $[\mathbf{C}]_{i,j}$. Then the Kronecker product between $\mathbf{C}$ and $\mathbf{D}$ is defined as 
 
 $$ \mathbf{C} \otimes \mathbf{D} := \begin{bmatrix} [\mathbf{C}]_{1,1} \mathbf{D} & \ldots & [\mathbf{C}]_{1, n_c} \mathbf{D} \\ \vdots & \ddots & \vdots \\ [\mathbf{C}]_{n_c, 1} \mathbf{D} & \ldots & [\mathbf{C}]_{n_c, n_c} \mathbf{D} \end{bmatrix} \; \in \mathbb{R}^{n_c n_d \times n_c n_d};$$
we refer to  \cite[Section 1.3.6]{Golub2012} for a survey on the properties of Kronecker  product.

First we define an approximate version $\widehat{\mathbf{A}}\in \mathbb{R}^{\mathcal{N}_{dof} \times \mathcal{N}_{dof} }$ of the matrix ${\mathbf{A}}$. In particular,
following the same ideas developed in \cite{MONTARDINI2018} for the Stokes system, we replace each local matrix $\mathbf{A}^{(k)}$ in \eqref{eq:block_A} with the block-diagonal matrix  $\widehat{\mathbf{A}}^{(k)}\in  \mathbb{R}^{\mathcal{N}_{dof}^{(k)}\times \mathcal{N}_{dof}^{(k)}}$, defined as
\begin{equation*}
\mathbf{\widehat{A}} ^{(k)}:=\begin{bmatrix}
 \mathbf{\widehat{A}}^{(k)} _{1} &  &  \\
& \ddots & \\
  & & \mathbf{\widehat{A}}^{(k)}_{d}
\end{bmatrix},
\end{equation*}
where $ \mathbf{\widehat{A}}^{(k)}_l\in\mathbb{R}^{n^{(k)}\times n^{(k)}} $ corresponds to $ \mathbf{A}^{(k)} _{l,l}$ in \eqref{eq:le_local_matrix}, but discretized in the parametric domain $\widehat{\Omega}$,   i.e,  referring to Section \ref{sec:multipatch} for the notation,
\begin{equation}
\label{eq:prec_inex_blocks}
[\mathbf{\widehat{A}}^{(k)} _{l}]_{i,j}:=\widehat{a}(\widehat{B}^{(k)}_{i,p}\mathbf{e}_l, \widehat{B}^{(k)}_{j,p}\mathbf{e}_l )
\quad \text{ for } i,j=1,\dots,n^{(k)},\end{equation}
\B  where $\mathbf{e}_l$ is the $l$-th vector of the canonical basis and
$$
 \widehat{a}(\vect{w},\vect{v}) :=  2 \mu \int_{\widehat{\Omega}} \varepsilon(\vect{w}) : \varepsilon(\vect{v}) \dFx + \lambda \int_{\widehat{\Omega}} \left(\nabla \cdot \vect{w}\right)\left( \nabla \cdot \vect{v} \right)\dFx.   
$$ 
The matrices $\mathbf{\widehat{A}}^{(k)} _{l}$ for $k=1,\dots,\numpatch$ and $l=1,\dots,d$ have a tensor product structure and, in  particular,  for $d =3$, that is the case we address in our numerical tests, we have 
\begin{align*}
\mathbf{\widehat{A}}^{(k)}_{1} & := \mu   {\widehat{K}}^{(k)}_3 \otimes  {\widehat{M}}_2^{(k)} \otimes  {\widehat{M}}^{(k)}_1
 + \mu   {\widehat{M}}^{(k)}_3 \otimes {\widehat{K}}_2^{(k)} \otimes {\widehat{M}}^{(k)}_1 + (2\mu + \lambda)  {\widehat{M}}^{(k)}_3 \otimes {\widehat{M}}^{(k)}_2 \otimes {\widehat{K}}^{(k)}_1, \\
 \mathbf{\widehat{A}}^{(k)}_{2} & := \mu {\widehat{K}}^{(k)}_3 \otimes {\widehat{M}}_2^{(k)} \otimes {\widehat{M}}^{(k)}_1
 + (2\mu + \lambda) {\widehat{M}}^{(k)}_3 \otimes {\widehat{K}}_2^{(k)} \otimes {\widehat{M}}^{(k)}_1  + \mu {\widehat{M}}^{(k)}_3 \otimes {\widehat{M}}^{(k)}_2 \otimes {\widehat{K}}^{(k)}_1, \\
\mathbf{\widehat{A}}^{(k)}_{3} & := (2\mu + \lambda) {\widehat{K}}^{(k)}_3 \otimes {\widehat{M}}_2^{(k)} \otimes {\widehat{M}}^{(k)}_1
 + \mu  {\widehat{M}}^{(k)}_3 \otimes {\widehat{K}}_2^{(k)} \otimes {\widehat{M}}^{(k)}_1  + \mu  {\widehat{M}}^{(k)}_3 \otimes {\widehat{M}}^{(k)}_2 \otimes {\widehat{K}}^{(k)}_1,
\end{align*} 
where $\widehat{{K}}_l^{(k)}$ and $\widehat{{M}}_l^{(k)}$ for $l=1,2,3$ are the local univariate stiffness and mass matrices. We remark that  in the construction of the matrices $ \widehat{K}^{(k)}_l$ and $\widehat{M}^{(k)}_l$ for $l = 1,\ldots,d$ we consider all the local degrees-of-freedom and thus also the interface degrees-of-freedom:
 the univariate stiffness matrices are always singular, which in turn means that also $\mathbf{\widehat{A}}^{(k)}$ is singular.
 
Similarly, we approximate each mass matrix ${\mathbf{M}}^{(k)}$ defined in \eqref{eq:le_local_MS} with the corresponding  mass matrix $
\mathbf{\widehat{{M}}}^{(k)}\in \mathbb{R}^{\mathcal{N}_{dof}^{(k)}\times \mathcal{N}_{dof}^{(k)}}$ in the parametric domain, that is defined component-wise  as
$$
\mathbf{\widehat{{M}}}^{(k)}:=\begin{bmatrix}
\mathbf{\widehat{M}}^{(k)}_1 & & \\
     &\ddots & \\
      & & \mathbf{\widehat{{M}}}^{(k)}_d
\end{bmatrix},
$$
where   $\mathbf{\widehat{M}}^{(k)}_l:= {\widehat{M}}^{(k)}_d \otimes \dots   \otimes {\widehat{M}}^{(k)}_1$ for $  l=1,\dots,d$ and $\widehat{M}^{(k)}_l$ are the univariate mass matrices in the $l$-th parametric direction. In particular, for $d=3$ we have that 
\begin{equation*}
\mathbf{\widehat{{M}}}_l^{(k)} :=  
{\widehat{M}}^{(k)}_3 \otimes {\widehat{M}}^{(k)}_2 \otimes {\widehat{M}}^{(k)}_1  \quad l=1,2,3.
\end{equation*}
We then define for $k=1,\dots,\numpatch  $
\begin{equation*}
   \mathbf{{P}}_{A}^{\mathrm{inex},(k)}  :=\begin{bmatrix}
\left (H^{(k)} \right )^{d-2}\left
  ( \mathbf{\widehat{A}}_1^{(k)}+ \mathbf{\widehat{{M}}}_1^{(k)}
\right ) & &\\
   & \ddots &\\
   &  & \left (H^{(k)} \right )^{d-2}\left
  ( \mathbf{\widehat{A}}_d^{(k)}+ \mathbf{\widehat{{M}}}_d^{(k)}
\right )
\end{bmatrix}   .
\end{equation*}
Note that, differently from $\widehat{\mathbf{A}}^{(k)}$, the matrices $\mathbf{\widehat{P}}_{A} ^{\mathrm{inex},(k)}$ are always positive definite, thanks to the addition of the mass matrix.

We also  introduce the   local Schur complement matrices $\mathbf{\widehat{S}}^{(k)}\in \mathbb{R}^{\mathcal{N}_{\Gamma}^{(k)}\times \mathcal{N}_{\Gamma}^{(k)}} $ associated to the matrices  $ \mathbf{\widehat{A}}^{(k)}$   for $k=1,\dots,\numpatch$, obtained by eliminating the interior degrees-of-freedom and defined as
\begin{equation}
\label{eq:prec_S_le}
\mathbf{\widehat{S}}^{(k)}:=       \mathbf{\widehat{A}}^{(k)}  _{\Gamma \Gamma }- \mathbf{\widehat{A}}^{(k)}  _{\Gamma  I}\left(\mathbf{\widehat{A}}^{(k)}  _{I I}\right)^{-1} \mathbf{\widehat{A}}^{(k)} _{I \Gamma }.
\end{equation}
Component-wise, the matrices \eqref{eq:prec_S_le} for $k=1,\dots,\numpatch$ can be split as    
$$
\mathbf{\widehat{S}}^{(k)}:=\begin{bmatrix}
\mathbf{\widehat{S}}_{1}^{(k)} & & \\
 & \ddots & \\
   & &\mathbf{\widehat{S}}_{d}^{(k)}
\end{bmatrix},
$$
where we defined for $l=1,\dots, d$
\begin{equation}
\label{eq:shur_comp}
{\widehat{\mathbf{S}}}_l^{(k)}:=      \left(\mathbf{\widehat{A}}^{(k)}_l \right)_{ \Gamma \Gamma }-\left(\mathbf{\widehat{A}}_l^{(k)} \right)_{\Gamma  I}\left[\left(\mathbf{\widehat{A}}^{(k)}_l \right)_{I I}\right]^{-1}\left(\mathbf{\widehat{A}}^{(k)}_l \right)_{I \Gamma}.
\end{equation}

In conclusion,   the inexact choice that we propose for $\mathbf{P}_A$
and   $\mathbf{P}_S$, respectively, is  
\begin{align} 
\mathbf{P}_{A}^{\mathrm{inex}}& :=\begin{bmatrix}
 \mathbf{{P}}_{A}^{\mathrm{inex},(1)}  & & \\
  & \ddots & \\
    & &  \mathbf{{P}}_{A}^{\mathrm{inex},(\numpatch)} 
 \end{bmatrix}, \label{eq:inexact_AP}\\ 
 \mathbf{P}_{S}^{\mathrm{inex}}& := \begin{bmatrix}
 \left (H^{(1)} \right )^{d-2} \mathbf{B}_{\Gamma}^{(1)} \mathbf{\widehat{S}}^{(1)}\left(\mathbf{B}^{(1)}_{\Gamma}\right)^T & & \\
  & \ddots & \\
    & &  \left (H^{(\numpatch)} \right )^{d-2} \mathbf{B}_{\Gamma}^{(\numpatch)}\mathbf{\widehat{S}}^{(\numpatch)}\left( \mathbf{B}_{\Gamma}^{(\numpatch)}\right)^T
 \end{bmatrix},\label{eq:inexact_SP}
\end{align}  
 where, for $k=1,\dots,\numpatch$,  $\mathbf{B}_{\Gamma}^{(k)}$ denotes the restriction of $\mathbf{B}_{\Gamma}$ to the $k$-th patch degrees-of-freedom.


\subsubsection{ Fast algorithm  }  
\label{sec:precapp}

At each iteration of the  preconditioned MINRES method  we have to compute the application  of the preconditioner   $\mathcal{B}^{-1}$   with the choice   $\mathbf{P}_A=\mathbf{P}_A^{\mathrm{inex}} $ and $\mathbf{P}_S = \mathbf{P}_S^{\mathrm{inex}}$, defined in \eqref{eq:inexact_AP}-\eqref{eq:inexact_SP}.   The computational core of the operation above   is the   computation of the solution of linear systems with matrices   $\widehat{\mathbf{P}}_A $ and   $\widehat{\mathbf{A}}_{II}$ and in particular of their local component-wise blocks   $\left (H^{(k)} \right )^{d-2}\left
  ( \mathbf{\widehat{A}}_l^{(k)}+ \mathbf{\widehat{{M}}}_l^{(k)}
\right )$  and $(\mathbf{\widehat{A}}_l^{(k)})_{II}$ for $l=1,\dots,d$ and  for $k=1,\ldots,\numpatch$. We  use   the FD method, for which we give full details in what follows.  

Following~\cite{MR3572372}, for 
a given patch $\Omega^{(k)}$ we consider   the eigendecompositions of the pencils $(\widehat{K}^{(k)}_l, \widehat{M}^{(k)}_l)$ and $((\widehat{K}^{(k)}_l)_{II}, (\widehat{M}^{(k)}_l)_{II})$: we find two couples of matrices $({D}^{(k)}_l$, $ {U}^{(k)}_l)$ and $(\widetilde{D}^{(k)}_l$, $ {\widetilde{U}}^{(k)}_l)$ such that 
\begin{equation}
\label{eq:eigendec_pa}
\widehat{K}^{(k)}_l = ({U}^{(k)}_l)^{-T} {D}^{(k)}_l ({U}^{(k)}_l)^{-1} \quad \text{ and } \quad \widehat{M}^{(k)}_l = ( {U}^{(k)}_l)^{-T} ( {U}^{(k)}_l)^{-1},
\end{equation}
\begin{equation*}
(\widehat{K}^{(k)}_l)_{II} = (\widetilde{U}^{(k)}_l)^{-T} \widetilde{D}^{(k)}_l (\widetilde{U}^{(k)}_l)^{-1} \quad \text{ and } \quad (\widehat{M}^{(k)}_l)_{II} = ( \widetilde{U}^{(k)}_l)^{-T} ( \widetilde{U}^{(k)}_l)^{-1},
\end{equation*}
where ${D}^{(k)}_l$ and $\widetilde{D}^{(k)}_l$ are diagonal matrices containing the generalized eigenvalues,    while the columns of $ {U}^{(k)}_l$ and $\widetilde{U}^{(k)}_l$ contain the corresponding normalized eigenvectors, respectively.
Then,  
  $\left (H^{(k)} \right )^{d-2}\left
  ( \mathbf{\widehat{A}}_l^{(k)}+ \mathbf{\widehat{{M}}}_l^{(k)}
\right )$  
can be rewritten   in this form
\begin{equation}
\label{eq:fac_ai}
  \left (H^{(k)} \right )^{d-2}\left
  ( \mathbf{\widehat{A}}_l^{(k)}+ \mathbf{\widehat{{M}}}_l^{(k)}
\right )    =({U}^{(k)}_d\otimes\dots\otimes{U}^{(k)}_1)^{-T}\mathbf{\Lambda}_l^{(k)} ({U}^{(k)}_d\otimes\dots\otimes {U}^{(k)}_1)^{-1}
\end{equation}
where $\mathbf{\Lambda}_l^{(k)}$ are diagonal matrices for $l=1,\dots,d$, that, e.g. in the case $d=3$, are defined as  
\begin{align*}
\mathbf{\Lambda}_1^{(k)}&=  {\left(H^{(k)}\right)^{d-2}} \left[ \mathbf{I}_{m_3^{(k)}}\otimes\mathbf{I}_{m_2^{(k)}}\otimes L^{(k)}_1 + \mu  \mathbf{I}_{m_3^{(k)}}\otimes D_2^{(k)}\otimes\mathbf{I} _{m_1^{(k)}}+ \mu D_3^{(k)}\otimes\mathbf{I}_{m_2^{(k)}}\otimes\mathbf{I}_{m_1^{(k)}} \right],\\ 
\mathbf{\Lambda}_2^{(k)}& = {\left(H^{(k)}\right)^{d-2}} \left[ \mu  \mathbf{I}_{m_3^{(k)}}\otimes\mathbf{I}_{m_2^{(k)}}\otimes D_1^{(k)} + \mathbf{I}_{m_3^{(k)}}\otimes  L^{(k)}_2\otimes \mathbf{I}_{m_1^{(k)}}  +   \mu D_3^{(k)}\otimes\mathbf{I}_{m_2^{(k)}}\otimes\mathbf{I}_{m_1^{(k)}}\right],\\
\mathbf{\Lambda}_3^{(k)}& = {\left(H^{(k)}\right)^{d-2}} \left[ \mu  \mathbf{I}_{m_3^{(k)}}\otimes\mathbf{I}_{m_2^{(k)}}\otimes D_1^{(k)}  + \mu  \mathbf{I}_{m_3^{(k)}}\otimes D_2^{(k)}\otimes\mathbf{I}_{m_1^{(k)}} +L^{(k)}_3\otimes  \mathbf{I}_{m_2^{(k)}}\otimes\mathbf{I}_{m_1^{(k)}}\right],
\end{align*}  
  where we defined the diagonal matrices 
$
L^{(k)}_l:=  (2\mu+\lambda)D_l^{(k)}+ \mathbf{I}_{m_l^{(k)}}  
$ for $l=1,2,3$.
 
Similarly, we have for the local  interior matrix $(\widehat{\mathbf{A}}_l^{(k)})_{II}$ the factorization
\begin{equation}
\label{eq:fac_schur}
(\widehat{\mathbf{A}}_l^{(k)})_{II}= (\widetilde{U}^{(k)}_d\otimes\dots\otimes\widetilde{U}^{(k)}_1)^{-T}\widetilde{\mathbf{\Lambda}}_l^{(k)} (\widetilde{U}^{(k)}_d\otimes\dots\otimes\widetilde{U}^{(k)}_1)^{-1}
\end{equation}
where $\widetilde{\mathbf{\Lambda}}_l^{(k)}$ are diagonal matrices for $l=1,\dots,d$, that, in the case $d=3$,  are defined as 
\begin{align*}
\widetilde{\mathbf{\Lambda}}_1^{(k)}&= (2\mu+\lambda) \mathbf{I}_{m_3^{(k)}-2}\otimes\mathbf{I}_{m_2^{(k)}-2}\otimes  \widetilde{D}_1^{(k)}    + \mu  \mathbf{I}_{m_3^{(k)}-2}\otimes \widetilde{D}_2^{(k)}\otimes\mathbf{I}_{m_1^{(k)}-2} + \mu\widetilde{D}_3^{(k)}\otimes\mathbf{I}_{m_2^{(k)}-2}\otimes\mathbf{I}_{m_1^{(k)}-2} ,\\ 
\widetilde{\mathbf{\Lambda}}_2^{(k)}& = \mu    \mathbf{I}_{m_3^{(k)}-2}\otimes\mathbf{I}_{m_2^{(k)}-2}\otimes \widetilde{D}_1^{(k)} +  (2\mu+\lambda)\mathbf{I}_{m_3^{(k)}-2}\otimes  \widetilde{D}_2^{(k)} \otimes \mathbf{I}_{m_1^{(k)}-2}  +   \mu\widetilde{D}_3^{(k)}\otimes\mathbf{I}_{m_2^{(k)}-2}\otimes\mathbf{I}_{m_1^{(k)}-2} ,\\
\widetilde{\mathbf{\Lambda}}_3^{(k)}& = \mu  \mathbf{I}_{m_3^{(k)}-2}\otimes\mathbf{I}_{m_2^{(k)}-2}\otimes \widetilde{D}_1^{(k)}  + \mu  \mathbf{I}_{m_3^{(k)}-2}\otimes \widetilde{D}_2^{(k)}\otimes\mathbf{I}_{m_1^{(k)}-2} +  (2\mu+\lambda)\widetilde{D}_3^{(k)}\otimes\mathbf{I}_{m_2^{(k)}-2}\otimes\mathbf{I}_{m_1^{(k)}-2} .
\end{align*} 
 The  direct inversion  of \eqref{eq:fac_ai} and \eqref{eq:fac_schur}   can be efficiently computed with the FD method ~\cite{MR3572372}. We detail below the algorithm for the solution of the system
$$
  \left (H^{(k)} \right )^{d-2}\left
  ( \mathbf{\widehat{A}}_l^{(k)}+ \mathbf{\widehat{{M}}}_l^{(k)}
\right )    \  \mathbf{s}=({U}^{(k)}_d\otimes\dots\otimes{U}^{(k)}_1)^{-T}\mathbf{\Lambda}_l^{(k)} ({U}^{(k)}_d\otimes\dots\otimes {U}^{(k)}_1)^{-1}\mathbf{s} =\mathbf{r}.
$$
The same algorithm, with obvious modifications, can be used for the solution of a system with matrix $(\widehat{\mathbf{A}}_{l}^{(k)})_{II}$.
\begin{algorithm}[H]
\caption{  FD method }\label{al:direct_P}
\begin{algorithmic}[1]
\State Compute the generalized eigendecomposition  \eqref{eq:eigendec_pa}.
\State Compute $\widetilde{\mathbf{r}} =   {(  U_d^{(k)} \otimes \dots\otimes U_1^{(k)})^T \mathbf{r}}$.
\State Compute $\widetilde{\mathbf{q}} = \left(\mathbf{\Lambda}_l^{(k)}\right)^{-1} \widetilde{\mathbf{r}}. $
\State Compute $\mathbf{s} =    ( U_d^{(k)} \otimes \dots\otimes U_1^{(k)})\ \widetilde{\mathbf{q}}. $
\end{algorithmic}
\end{algorithm}

We now   discuss the computational cost of Algorithm \ref{al:direct_P}, and for this purpose we assume for simplicity that    the matrices $\widehat{K}^{(k)}_l$ and $\widehat{M}_l^{(k)}$ have the same order $n$ for $l=1,\dots,d$, i.e. $\mathcal{N}_{dof}^{(k)}=n^d$.
Step 1, that represents the setup of the preconditioner and can therefore be performed only once, requires $O(d n^{3})$ FLOPs, which is optimal (i.e. proportional to the number of degrees-of-freedom $\mathcal{N}_{dof}^{(k)}$) for $d=3$. Step 3 involves the inversion of a diagonal matrix, which always yields an optimal cost. The leading cost of Algorithm \ref{al:direct_P} is given by Step 2 and Step 4, each of which can be rewritten in terms of $d$ dense matrix-matrix products, and require a total of $4d n \mathcal{N}_{dof}^{(k)}$ FLOPs. Although this cost is slightly sub-optimal, thanks to the high efficient implementation of dense matrix-matrix products in modern computers, in practice
the computational time spent by Algorithm \ref{al:direct_P} turns out to be negligible in the overall Krylov method, up to an  very   large number of degrees-of-freedom
 (see~\cite[Table 10]{MR3572372} and \cite{montardini2018space}).
  

\subsubsection{Spectral estimates  }
\label{sec:specest}

In this section, we prove the spectral estimates \eqref{eq:bound_A} and \eqref{eq:bound_S}   for the choice   $\mathbf{P}_A=\mathbf{P}_A^{\mathrm{inex}} $ and $\mathbf{P}_S = \mathbf{P}_S^{\mathrm{inex}}$,  defined in \eqref{eq:inexact_AP}-\eqref{eq:inexact_SP}.  
In particular, we will show  that the bounds \eqref{eq:bound_A} and \eqref{eq:bound_S} hold with constants independent of   $h, H$   and $p$  by restricting to a single patch $\Omega^{(k)}$ with  $k=1,\ldots,\numpatch$.
The final result is summarized in Proposition~\ref{thm:main}.

\begin{lemma}
\label{l:le_A0}
There are positive constant $\tilde{a}^{(k)}_0$ and
$\tilde{a}^{(k)}_1$, independent of $h$,   $H$ and $p$, such that 
\begin{eqnarray}
\label{eq:le_A0_lower}
 {\left(H^{(k)}\right)^{d-2}}  \mathbf{u}^T \widehat{\mathbf{A}}^{(k)} \mathbf{u} & \geq & \tilde{a}^{(k)}_0 \mathbf{u}^T \mathbf{A}^{(k)} \mathbf{u}, \qquad \forall \,  \mathbf{u} \in \mathbb{R}^{\mathcal{N}^{(k)}_{dof}}, \\
	\label{eq:le_A0_upper}
  {\left(H^{(k)}\right)^{d-2}}  \mathbf{u}^T \widehat{\mathbf{A}}^{(k)} \mathbf{u} & \leq & \tilde{a}^{(k)}_1 \mathbf{u}^T \mathbf{A}^{(k)} \mathbf{u}, \qquad \forall \,  \mathbf{u} \in range(\mathbf{A}^{(k)}).
	\end{eqnarray} 
\end{lemma} 

\begin{proof}
Let $\vect{u}_h \in \mathcal{V}^{(k)}_h$ and let $\mathbf{u} \in  \mathbb{R}^{\mathcal{N}^{(k)}_{dof}}$ be its coordinate vector. Let also  $\vect{\widehat{u}}_h := \vect{u}_h \circ \mathcal{F}^{(k)}  $ and let $\vect{\widehat{u}}_{h,l}$ for $l=1,\dots, d$ denote its  components, i.e. $\vect{\widehat{u}}_{h}=\sum_{l=1}^d \vect{\widehat{u}}_{h,l}\mathbf{e}_l$. Note that, using \eqref{eq:prec_inex_blocks}, we have that 
\begin{equation}
\label{eq:estimate_a}
 \mu| \widehat{\vect{u}}_h|^2_{H^1(\widehat{\Omega})}= \sum_{l=1}^d \mu| \widehat{\vect{u}}_{h,l} |^2_{H^1(\widehat{\Omega})}\leq \sum_{l=1}^d \left(\mu| \widehat{\vect{u}}_{h,l} |^2_{H^1(\widehat{\Omega})}+    ( \lambda + \mu )\left\| \frac{\partial}{\partial \widehat{\vect{x}}_l}\widehat{\vect{u}}_{h,l}\right\|^2_{L^2(\widehat{\Omega})}\right)= \mathbf{u}^T \widehat{\mathbf{A}}^{(k)}\mathbf{u}.
\end{equation}
  Denoting by $\|\cdot\|_2$  the norm induced by the Euclidean vector norm and    using the definition of  $a(\cdot, \cdot)$ in \eqref{eq:bil_lin} and
the estimate   \eqref{eq:estimate_a},    we get

\begin{eqnarray*}  
\mathbf{u}^T \mathbf{A}^{(k)}\mathbf{u} & = & 2\mu \|\varepsilon(\vect{u}_h)\|_{L^2(\Omega^{(k)})}^2 + \lambda \left\| \nabla \cdot \vect{u}_h \right\|_{L^2(\Omega^{(k)})}^2 \nonumber \\
& \leq &(2\mu  + d\lambda)|\vect{u}_h|_{H^1(\Omega^{(k)})}^2  \nonumber \\ 
& \leq &  \frac{ 2\mu  + d\lambda}{\mu}\sup_{\widehat{\mathbf x}\in\widehat{\Omega}}\left\{ \left|\mathrm{det}(J_{\mathcal{F}^{(k)}}(\widehat{\mathbf x})) \right| \left\| J_{\mathcal{F}^{(k)}}^{-1}(\widehat{\mathbf x})\right\|_2^2\right\}  \mu | \vect{\widehat{u}}_h |_{H^1(\widehat{\Omega})}^2  \\
& \leq & \frac{ 2\mu  + d\lambda}{\mu}\sup_{\widehat{\mathbf
         x}\in\widehat{\Omega}}\left\{ \left|\mathrm{det}\left
         (\left(H^{(k)}\right)^{-1} J_{\mathcal{F}^{(k)}}(\widehat{\mathbf x})
         \right) \right| \left\| H^{(k)}
         J_{\mathcal{F}^{(k)}}^{-1}(\widehat{\mathbf
         x})\right\|_2^2\right\} \left(H^{(k)}\right)^{d-2}   \mu | \vect{\widehat{u}}_h |_{H^1(\widehat{\Omega})}^2  \\
& \leq &     \frac{ {\left(H^{(k)}\right)^{d-2}}}{\widetilde{a}_0^{(k)}} \mathbf{u}^T  \widehat{\mathbf{A}}^{(k)}\mathbf{u}  \nonumber
\end{eqnarray*} 
where $ \displaystyle \frac{1}{\tilde{a}_0^{(k)}}:= \frac{ 2\mu  + d\lambda}{\mu}\sup_{\widehat{\mathbf
         x}\in\widehat{\Omega}}\left\{ \left|\mathrm{det}\left
         (\left(H^{(k)}\right)^{-1} J_{\mathcal{F}^{(k)}}(\widehat{\mathbf x})
         \right) \right| \left\| H^{(k)}
         J_{\mathcal{F}^{(k)}}^{-1}(\widehat{\mathbf
         x})\right\|_2^2\right\}  $, and inequality
   \eqref{eq:le_A0_lower} is proven. The quantity
   \begin{displaymath}
     \sup_{\widehat{\mathbf
         x}\in\widehat{\Omega}}\left\{ \left|\mathrm{det}\left
         ((H^{(k)})^{-1} J_{\mathcal{F}^{(k)}}(\widehat{\mathbf x})
         \right) \right| \left\| H^{(k)}
         J_{\mathcal{F}^{(k)}}^{-1}(\widehat{\mathbf
         x})\right\|_2^2\right\} 
   \end{displaymath}
depends on $\left(H^{(k)}\right)^{-1} J_{\mathcal{F}^{(k)}}$ and then depends on
the \emph{shape} of the patch $\Omega^{(k)}$ but is independent of its
actual diameter  $H^{(k)}$, that is, it is  invariant for homothety
transformations of $\Omega^{(k)}$. In this sense,  $\tilde{a}_0^{(k)}$
is independent of $H^{(k)} $, and of course independent of $p$ and of
$h^{(k)} $. \B

We now turn to the proof of bound \eqref{eq:le_A0_upper}.   As above,    $\mathbf{u}\in \textrm{range}(\mathbf{A}^{(k)})$ is the vector that represents $\vect{u}_h\in \mathcal{V}^{(k)}_h$  and  $\widehat{\vect{u}}_h:=\vect{u}_h\circ \mathcal{F}^{(k)}.$ In this case, we  use the Korn inequality (see~\cite[Lemma 4]{MR1787293}) 
\begin{equation*}  C_{\mathrm{korn}}^{(k)} |\vect{u}_h|_{H^1(\Omega^{(k)})}^2 \; \leq \; \|\varepsilon(\vect{u}_h)\|_{L^2(\Omega^{(k)})}^2, \end{equation*}
where $C_{\mathrm{korn}}^{(k)}$ is a positive constant that depends only on $\Omega^{(k)}$. Thus  
\begin{eqnarray*} 
 {\left(H^{(k)}\right)^{d-2}}  \B\mathbf{u}^T  \widehat{\mathbf{A}}^{(k)}\mathbf{u}   &  \leq &   {\left(H^{(k)}\right)^{d-2}}   \left( 2\mu \|\varepsilon(\widehat{\vect{u}}_h)\|_{L^2(\widehat{\Omega})}^2+\lambda |\widehat{\vect{u}}_h|^2_{H^1(\widehat{\Omega})} \right)  \leq    {\left(H^{(k)}\right)^{d-2}}    (2\mu + \lambda)  |\widehat{\vect{u}}_h|^2_{H^1(\widehat{\Omega})}\\
& \leq &    \sup_{\widehat{\mathbf
         x}\in\widehat{\Omega}}\left\{\left|\ \mathrm{det}\left(  {H^{(k)}}    J^{-1}_{ \mathcal{F}^{(k)} }(\widehat{\mathbf x})\right) \right| \left\| {  {\left(H^{(k)}\right)^{-1}}  }  J_{ \mathcal{F}^{(k)} }(\widehat{\mathbf x})\right\|_2^2\right\} (2\mu + \lambda)|\vect{u}_h|_{H^1(\Omega^{(k)})}^2  \nonumber \\ 
& \leq &      \sup_{\widehat{\mathbf
         x}\in\widehat{\Omega}}\left\{\left|\ \mathrm{det}\left(  {H^{(k)}}    J^{-1}_{ \mathcal{F}^{(k)} }(\widehat{\mathbf x})\right) \right| \left\| {  {\left(H^{(k)}\right)^{-1}}  }  J_{ \mathcal{F}^{(k)} }(\widehat{\mathbf x})\right\|_2^2\right\} \frac{2\mu+\lambda}{2\mu  C^{(k)}_{\mathrm{korn}}} 2 \mu \|\varepsilon(\vect{u}_h)\|_{L^2(\Omega^{(k)})}^2  \\
& \leq &\widetilde{a}_1^{(k)} \left(  2 \mu \|\varepsilon(\vect{u}_h)\|_{L^2(\Omega^{(k)})}^2 + \lambda\left\|\nabla \cdot  \vect{u}_h \right\|^2_{L^2(\Omega^{(k)})} \right) \\
 & = & \widetilde{a}_1^{(k)} \mathbf{u}^T {\mathbf{A}}^{(k)}\mathbf{u}, \nonumber
\end{eqnarray*}
where $\widetilde{a}_1^{(k)}:=  \displaystyle \sup_{\widehat{\mathbf
         x}\in\widehat{\Omega}}\left\{\left|\ \mathrm{det}\left( {H^{(k)}}    J^{-1}_{ \mathcal{F}^{(k)} }(\widehat{\mathbf x})\right) \right| \left\|  {  {\left(H^{(k)}\right)^{-1}}  }   J_{ \mathcal{F}^{(k)} }(\widehat{\mathbf x}) \right\|_2^2\right\} \frac{2\mu + \lambda}{2\mu C^{(k)}_{\mathrm{korn}}}  $.
\end{proof}
 
\begin{lemma}
\label{l:le_A}
There are positive constants $ a^{(k)}_0$ and $a^{(k)}_1$, independent of $h,   H   $ and $p$, such that  for all $ \mathbf{u}\in \textrm{range}(\mathbf{A}^{(k)})$
\begin{equation*} 
 {a}^{(k)}_0\mathbf{u}^T\left(\mathbf{A}^{(k)} +   \left(H^{(k)}\right)^{-2}   \mathbf{M}^{(k)}\right)\mathbf{u} \; \leq \; \mathbf{u}^T\widehat{\mathbf{P}}^{(k)}_A\mathbf{u} \; \leq \; {a}^{(k)}_1\mathbf{u}^T\left(\mathbf{A}^{(k)} +   \left(H^{(k)}\right)^{-2}   \mathbf{M}^{(k)}\right)\mathbf{u}. 
\end{equation*} 
\end{lemma}

\begin{proof}
Let $\vect{u}_h \in \mathcal{V}_h^{(k)}$ and let $\mathbf{u} \in  \mathbb{R}^{\mathcal{N}^{(k)}_{dof}}$ be its coordinate vector. Let also $\vect{\widehat{u}}_h := \vect{u}_h \circ \mathcal{F}^{(k)}$.
It holds
\begin{align*}
  \left(H^{(k)}\right)^{-2}   \mathbf{u}^T\mathbf{M}^{(k)}\mathbf{u} \; & = \;  \left(H^{(k)}\right)^{-2}   \int_{{\Omega}^{(k)}} \vect{{u}}_h^2 \dx \; = \;   \left(H^{(k)}\right)^{-2} \B\int_{\widehat{\Omega}} \vect{\widehat{u}}_h^2|\mathrm{det}(J_{\mathcal{F}^{(k)}})| \dFx \; \\
& \leq \; \sup_{\widehat{\mathbf
         x}\in\widehat{\Omega}}\left\{\left| \mathrm{det}\left(  \left(H^{(k)}\right)^{-1}   J_{\mathcal{F}^{(k)}}(\widehat{\mathbf{x}}) \right)\right| \right\}  \left(H^{(k)}\right)^{d-2}  \int_{\widehat{\Omega}} \vect{\widehat{u}}_h^2 \dFx 
\end{align*}
and in a similar way
\begin{equation*}
  \left(H^{(k)}\right)^{-2}  \mathbf{u}^T\mathbf{M}^{(k)}\mathbf{u} \; \geq \;
 \inf_{\widehat{\mathbf{x}}\in\widehat{\Omega}}\left\{\left|\mathrm{det}\left( \left(H^{(k)}\right)^{-1}   J_{\mathcal{F}^{(k)}}(\widehat{\mathbf{x}})\right)\right| \right \}   \left(H^{(k)}\right)^{d-2}   \int_{\widehat{\Omega}} \vect{\widehat{u}}_h^2 \dFx. \end{equation*}
Therefore, we infer
\begin{equation}
\label{eq:mass_equiv}
   \left(H^{(k)}\right)^{-2}   m_0^{(k)} \mathbf{u}^T\mathbf{M}^{(k)}\mathbf{u} \; \leq \;   \left(H^{(k)}\right)^{d-2} \B\mathbf{u}^T\widehat{\mathbf{M}}^{(k)}\mathbf{u} \; \leq \;   \left(H^{(k)}\right)^{-2}   m^{(k)}_1\mathbf{u}^T\mathbf{M}^{(k)}\mathbf{u}, 
\end{equation}
with $\tfrac1{m_0^{(k)}} :=  \displaystyle \inf_{\widehat{\mathbf{x}}\in\widehat{\Omega}}\left\{\left|\mathrm{det}\left(  \left(H^{(k)}\right)^{-1}   J_{\mathcal{F}^{(k)}}(\widehat{\mathbf{x}})\right)\right| \right \} $ and $ \tfrac1{m_1^{(k)}} :=  \displaystyle \sup_{\widehat{\mathbf{x}}\in\widehat{\Omega}}\left\{\left|\mathrm{det}\left(  \left(H^{(k)}\right)^{-1}  J_{\mathcal{F}^{(k)}}(\widehat{\mathbf{x}})\right)\right| \right \} $.
 Analogously to what happens for  $\widetilde{a}_0^{(k)}$ and $\widetilde{a}_1^{(k)}$ in \eqref{eq:le_A0_lower} and \eqref{eq:le_A0_upper}, the constants $m_0^{(k)}$ and  $m_1^{(k)}$ depend only on $\left(H^{(k)}\right)^{-1}J_{\mathcal{F}^{(k)}}$ and thus on the shape of the patch $\Omega^{(k)}$ and they are independent of $H^{(k)}$, of $p$ and of $h^{(k)}$. 

The application of the inequality \eqref{eq:mass_equiv} and of Lemma \ref{l:le_A0} to a vector $\mathbf{u}\in \mathrm{range}(\mathbf{A}^{(k)})$ immediately give the required result.
\end{proof}

We now consider the estimates for $\mathbf{\widehat{S}}^{(k)}$.

\begin{lemma}
\label{l:Schur_estim_le}
Let $\mathbf{B}^{(k)}$ and $\mathbf{R}^{(k)}$ denote the restrictions  of $\mathbf{B}$  and $\mathbf{R}$, respectively, to the $k$-th patch degrees-of-freedom, and let  $\mathbf{G}^{(k)}:=\mathbf{B}^{(k)}\mathbf{R}^{(k)}$.
Then there are positive constants $s_0^{(k)}$ and $s_1^{(k)}$, independent of $h,  H  $ and $p$, such that $\forall \  \vect{\lambda}\in \mathrm{ker}\left((\mathbf{G}^{(k)})^T\right)$
\begin{equation*}
s_0^{(k)}\vect{\lambda}^T\mathbf{B}_{\Gamma}^{(k)}\mathbf{S}^{(k)}\left(\mathbf{B}_{\Gamma}^{(k)}\right)^T \vect{\lambda} \; \leq \;  {\left(H^{(k)}\right)^{d-2}} \vect{\lambda}^T\mathbf{B}_{\Gamma}^{(k)}\widehat{\mathbf{S}}^{(k)}\left(\mathbf{B}_{\Gamma}^{(k)}\right)^T\vect{\lambda} \; \leq \; s_1^{(k)}\vect{\lambda}^T\mathbf{B}_{\Gamma}^{(k)}\mathbf{S}^{(k)}\left(\mathbf{B}_{\Gamma}^{(k)}\right)^T \vect{\lambda}. 
\end{equation*} 
\end{lemma}

\begin{proof}

We first observe that
\begin{equation} \label{eq:lambda_theta} \vect{\lambda} \in \mathrm{ker}\left((\mathbf{G}^{(k)})^T\right) \Longleftrightarrow \left(\mathbf{B}_{\Gamma}^{(k)}\right)^T \vect{\lambda} \in \mathrm{range}(\mathbf{S}^{(k)}). \end{equation}
Indeed, $ (\mathbf{G}^{(k)})^T \vect{\lambda} = \left(\mathbf{R}^{(k)}\right)^T \left(\mathbf{B}^{(k)}\right)^T  \vect{\lambda}  = 0  $ if and only if $\left(\mathbf{B}^{(k)}\right)^T  \vect{\lambda} \in \mathrm{ker}\left((\mathbf{R}^{(k)})^T\right) = \mathrm{range}(\mathbf{A}^{(k)})$, and the latter is equivalent to $\left(\mathbf{B}_{\Gamma}^{(k)}\right)^T  \vect{\lambda} \in \mathrm{range}(\mathbf{S}^{(k)})$.
In light of \eqref{eq:lambda_theta}, the statement we want to prove is equivalent to 
\begin{equation}
\label{eq:A_S_le}
s_0^{(k)} \; \leq \;  {\left(H^{(k)}\right)^{d-2}} \frac{\vect{\theta}^T \widehat{\mathbf{S}}^{(k)} \vect{\theta}}{\vect{\theta}^T \mathbf{S}^{(k)} \vect{\theta}} \; \leq \; s_1^{(k)} \qquad \forall \  \vect{\theta}\in \mathrm{range}(\mathbf{S}^{(k)}).
\end{equation} 
To this end, referring to~\cite[Section 5.1]{MR1787293} for further details,  we use that 
\begin{equation*}
\vect{\theta}^T \mathbf{S}^{(k)} \vect{\theta} \; = \; \min_{\mathbf{v} \in  \mathbb{R}^{\mathcal{N}^{(k)}_{dof}}, \ \mathbf{v}_{|_{\Gamma}}  =  \vect{\theta}} \mathbf{v}^T \mathbf{A}^{(k)} \mathbf{v} ,  
\qquad
 \vect{\theta}^T\widehat{\mathbf{S}}^{(k)}  \vect{\theta} \; =     \; \min_{\mathbf{v} \in  \mathbb{R}^{\mathcal{N}^{(k)}_{dof}}, \ \mathbf{v}_{|_{\Gamma}} = \vect{\theta}} \mathbf{v}^T \widehat{\mathbf{A}}^{(k)} \mathbf{v},
\end{equation*}
  where $\mathbf{v}_{|_{\Gamma}}$ refers to the   degrees-of-freedom of $\mathbf{v}$ that belong to the  interface $\Gamma$.  
Thus, in particular 
\begin{equation*}
   {\left(H^{(k)}\right)^{d-2}} \frac{\vect{\theta}^T \widehat{\mathbf{S}}^{(k)}  \vect{\theta}}{\vect{\theta}^T \mathbf{S}^{(k)} \vect{\theta}}
\; = \;   {\left(H^{(k)}\right)^{d-2}}    \frac{ \displaystyle   \min_{\mathbf{v} \in \mathbb{R}^{\mathcal{N}^{(k)}_{dof}}, \ \mathbf{v} |_{\Gamma} = \vect{\theta}} \mathbf{v}^T\widehat{\mathbf{A}}^{(k)}  \mathbf{v}}{ \displaystyle \min_{\mathbf{v} \in  \mathbb{R}^{\mathcal{N}^{(k)}_{dof}}, \ \mathbf{v} |_{\Gamma} = \vect{\theta}} \mathbf{v}^T \mathbf{A}^{(k)} \mathbf{v}}.
\end{equation*}

Since inequality \eqref{eq:le_A0_lower} holds for every vector belonging to  $\mathbb{R}^{\mathcal{N}^{(k)}_{dof}}$, we infer that 
\begin{equation*} 
  \left( H^{(k)}\right)^{d-2}  \min_{\mathbf{v} \in \mathbb{R}^{\mathcal{N}^{(k)}_{dof}}, \ \mathbf{v}_{|_{\Gamma}} = \vect{\theta}} \mathbf{v}^T\widehat{\mathbf{A}}^{(k)}  \mathbf{v} \; \geq \; \tilde{a}_0^{(k)} \min_{\mathbf{v} \in  \mathbb{R}^{\mathcal{N}^{(k)}_{dof}}, \ \mathbf{v}_{|_{\Gamma}} = \vect{\theta}} \mathbf{v}^T \mathbf{A}^{(k)} \mathbf{v}, 
\end{equation*}
and the left inequality of \eqref{eq:A_S_le} holds by taking $s^{(k)}_0 := \tilde{a}^{(k)}_0$.

Now, if we define $ \mathbf{v}_{\mathrm{elast}} := \underset{\mathbf{v} \in \mathbb{R}^{\mathcal{N}^{(k)}_{dof}}, \ \mathbf{v}_{|_{\Gamma}} = \vect{\theta}} {\mathrm{argmin}}\mathbf{v}^T \mathbf{A}^{(k)}  \mathbf{v} $,  it can be shown that since $\vect{\theta} \in \mathrm{range}(\mathbf{S}^{(k)}) $, then $\mathbf{v}_{\mathrm{elast}}\in  \mathrm{range}(\mathbf{A}^{(k)})$. Thus, 
\begin{align*}
  \left( H^{(k)}\right)^{d-2}    \frac{ \displaystyle  \min_{\mathbf{v} \in \mathbb{R}^{\mathcal{N}^{(k)}_{dof}}, \ \mathbf{v}_{|_{\Gamma}} = \vect{\theta}} \mathbf{v}^T\widehat{\mathbf{A}}^{(k)}  \mathbf{v}}{ \displaystyle \min_{\mathbf{v} \in  \mathbb{R}^{\mathcal{N}^{(k)}_{dof}}, \ \mathbf{v}_{|_{\Gamma}} = \vect{\theta}} \mathbf{v}^T \mathbf{A}^{(k)} \mathbf{v}} 
\;&  = \;
  \left( H^{(k)}\right)^{d-2}    \frac{ \displaystyle  \min_{\mathbf{v} \in \mathbb{R}^{\mathcal{N}^{(k)}_{dof}}, \ \mathbf{v}_{|_{\Gamma}} = \vect{\theta}} \mathbf{v}^T\widehat{\mathbf{A}}^{(k)}  \mathbf{v}}{ \mathbf{v_{\mathrm{elast}}}^T \mathbf{A}^{(k)} \mathbf{v_{\mathrm{elast}}}} 
\; \\ & \leq \; \left( H^{(k)}\right)^{d-2}  
\frac{ \mathbf{v_{\mathrm{elast}}}^T  \widehat{\mathbf{A}}^{(k)} \mathbf{v_{\mathrm{elast}}} }{ \mathbf{v}_{\mathrm{elast}}^T \mathbf{A}^{(k)} \mathbf{v}_{\mathrm{elast}}} 
\;\leq \; 
\tilde{a}^{(k)}_1, 
\end{align*}
where the last inequality follows from \eqref{eq:le_A0_upper}. Thus, the right inequality of \eqref{eq:A_S_le} holds by taking $s^{(k)}_1 := \tilde{a}^{(k)}_1$.
\end{proof}

We are finally ready to prove the main result of this Section.
\begin{prop}
\label{thm:main} 
  The spectral estimates  \eqref{eq:bound_A} and \eqref{eq:bound_S}
  are satisfied for the choice $\mathbf{P}_A=
  \mathbf{P}_A^{\mathrm{inex}}$ and $\mathbf{P}_S=   \mathbf{P}_S^{\mathrm{inex}}$  defined in \eqref{eq:inexact_AP} and  \eqref{eq:inexact_SP}, respectively, with constants $a_0, a_1, s_0,$ and $s_1$ that are independent of $h, H$ and $p$. 
\end{prop}
\begin{proof}
The result immediately follows from  
Lemma \ref{l:le_A} and Lemma \ref{l:Schur_estim_le}  
by taking
$$ a_0 := \min_{k=1,\dots, \numpatch} \left\{a_0^{(k)}\right\}, \quad a_1 := \max_{k=1,\dots, \numpatch} \left\{a_1^{(k)}\right\},$$
$$  s_0 := \min_{k=1,\dots, \numpatch} \left\{s_0^{(k)}\right\}, \quad s_1 := \max_{k=1,\dots, \numpatch} \left\{s_1^{(k)}\right\}.$$
\end{proof}

\subsubsection{Inclusion of the geometry information }
\label{sec:geo_inclusion}
 We have shown in Section~\ref{sec:specest}
that the 
spectral estimates for the local preconditioned problems do not depend    on $h, H$  and on $p$. 
 However, they depend on the geometry parametrizations   $\mathcal{F}^{(k)}$.
 We propose a strategy that allows to include in   $\widehat{\mathbf{P}}_A^{(k)}$ and  $\widehat{\mathbf{P}}_S^{(k)}$ some information about the parametrization maps, while keeping the Kronecker structure of the local solvers. 
This strategy, which we briefly present below, is made of two steps:
an approximation of the coefficients and a diagonal scaling.   For more
details, see e.g. \cite[Appendix C]{montardini2018space}. 

 We consider the matrices on the diagonal blocks of the system matrix \eqref{eq:le_local_matrix}, and, referring to Section \ref{sec:preliminaries} for the notation of the basis functions, we rewrite their entries as integrals on the parametric domain $\widehat{\Omega}$ for $k=1,\dots,\numpatch$ and for $l=1,\dots,d$   as
\begin{align*}
[\mathbf{A}^{(k)}_{l,l}]_{i,j}& = 2 \mu \int_{{\Omega}^{(k)}} \varepsilon( \mathbf{e}_l{B}^{(k)}_{i,p}) : \varepsilon( \mathbf{e}_l{B}^{(k)}_{j,p})\d\vect{x} + \lambda \int_{{\Omega}^{(k)}} \nabla \cdot \left(\mathbf{e}_l{B}^{(k)}_{i,p} \right) \nabla \cdot \left( \mathbf{e}_l{B}^{(k)}_{j,p}\right) \d \vect{x} \\
 &  =  \int_{\widehat{\Omega}} \left(\nabla \widehat{B}^{(k)}_{i,p}\right)^T\ \mathfrak{C}^{(k)}_l  \nabla  \widehat{B}^{(k)}_{j,p}\d\widehat{\vect{x}}   \quad \text{for } i,j=1,\dots, n^{(k)},
\end{align*}
 where for $k=1,\dots,\numpatch$ and for $l=1,\dots,d$
\begin{equation*}
\mathfrak{C}^{(k)}_l=  \left[ \mu J_{\mathcal{F}^{(k)}}^{-1}J_{\mathcal{F}^{(k)}}^{-T} + (\mu+\lambda) J_{\mathcal{F}^{(k)}}^{-1}\mathbf{e}_l\mathbf{e}_l^TJ_{\mathcal{F}^{(k)}}^{-T}  \right]|\det(J_{\mathcal{F}^{(k)}})|.
\end{equation*} 
We then approximate each diagonal entry of  the matrices $\mathfrak{C}^{(k)}_l$ for $k=1,\dots, \numpatch$ and  $l=1,\dots, d$ as
\begin{equation*}
[ \mathfrak{C}_l^{(k)}(\vect{\eta})]_{i,i}^{(k)}\approx [ \widetilde{\mathfrak{C}}^{(k)}_l(\vect{\eta})]_{i,i}:= \beta_{l,1}^{(k)}(\eta_1)\dots \beta^{(k)}_{l,i-1}(\eta_{i-1})\nu^{(k)}_{l,i} (\eta_i)\beta^{(k)}_{l,i+1}(\eta _{i+1})\dots\beta^{(k)}_{l,d}(\eta_d)\ \ \ \mathrm{ for \ }  i = 1,\dots,d.
\end{equation*}
The approximation, which is performed using a straightforward variant
of the algorithm proposed in \cite{Wachspress1984,Diliberto1951} that
can be found in   \cite[Appendix C]{montardini2018space}, is computed
directly at the quadrature points.   The cost of this operation is
therefore proportional to the number of quadrature points: if we
assume that the number of elements is the same in each parametric
direction and that it is equal to $n_{el}^{(k)}$, then, using standard
Gauss quadrature rules, the approximation  cost is
$O(n_{el}^{(k)}p^d)$ FLOPs.  We could   reduce the  cost of this
procedure by computing the approximation on a coarser grid or adopting
a more efficient quadrature scheme.    Then we define for $k=1,\dots,\numpatch$ and for $l=1,\dots,d$
$$
[\widetilde{\mathbf{A}}^{(k)}_{l}]_{i,j} : =  \int_{\widehat{\Omega}} \left(\nabla \widehat{B}^{(k)}_{i,p}\right)^T\ \widetilde{\mathfrak{C}}^{(k)}_l  \nabla  \widehat{B}^{(k)}_{j,p}\d\widehat{\vect{x}}  \quad \text{ for } i,j=1,\dots,n^{(k)}
$$
and also 
$\widetilde{\mathbf{A}}^{(k)}:=\begin{bmatrix}
\widetilde{\mathbf{A}}^{(k)}_{1} & &\\
& \ddots  \\
& & \widetilde{\mathbf{A}}^{(k)}_{d}
\end{bmatrix}$ for $k=1,\dots,\numpatch$. 
We remark that the matrices $\widetilde{\mathbf{A}}^{(k)}_l$ maintain the same Kronecker structure as the matrices $\widehat{\mathbf{A}}^{(k)}_l$ in \eqref{eq:prec_inex_blocks} for $k=1,\dots, \numpatch$ and for $l=1,\dots,d$. In particular, for $d=3$ and $k=1,\dots,\numpatch$ we have 
\begin{align*}
\mathbf{\widetilde{A}}^{(k)}_{1} & :=    {\widetilde{K}}^{(k)}_{1,3} \otimes  {\widetilde{M}}_{1,2}^{(k)} \otimes  {\widetilde{M}}^{(k)}_{1,1}
 +   {\widetilde{M}}^{(k)}_{1,3} \otimes \ {\widetilde{K}}_{1,2}^{(k)} \otimes  {\widetilde{M}}^{(k)}_{1,1}  +    {\widetilde{M}}^{(k)}_{1,3} \otimes  {\widetilde{M}}^{(k)}_{1,2} \otimes  {\widetilde{K}}^{(k)}_{1,1}, \\
  \mathbf{\widetilde{A}}^{(k)}_{2} & :=    {\widetilde{K}}^{(k)}_{2,3} \otimes {\widetilde{M}}_{2,2}^{(k)} \otimes  {\widetilde{M}}^{(k)}_{2,1}
 +   {\widetilde{M}}^{(k)}_{2,3} \otimes  {\widetilde{K}}_{2,2}^{(k)} \otimes {\widetilde{M}}^{(k)}_{2,1} +   {\widetilde{M}}^{(k)}_{2,3} \otimes  {\widetilde{M}}^{(k)}_{2,2} \otimes  {\widetilde{K}}^{(k)}_{2,1}, \\
 \mathbf{\widetilde{A}}^{(k)}_{3} & :=    {\widetilde{K}}^{(k)}_{3,3} \otimes {\widetilde{M}}_{3,2}^{(k)} \otimes {\widetilde{M}}^{(k)}_{3,1}
 +   {\widetilde{M}}^{(k)}_{3,3} \otimes  {\widetilde{K}}_{3,2}^{(k)} \otimes {\widetilde{M}}^{(k)}_{3,1}   +    {\widetilde{M}}^{(k)}_{3,3} \otimes {\widetilde{M}}^{(k)}_{3,2} \otimes  {\widetilde{K}}^{(k)}_{3,1},
\end{align*}
where  for  $i,j=1,\dots, m_{l}^{(k)}$ and for $l,m=1,\dots, d$
\begin{equation*}
[\widetilde{{K}}_{l,m}^{(k)}]_{i,j}:=\int_{0}^1\nu_{l,m}^{(k)}(\eta_m)\widehat{b}_{i,p}'(\eta_m)\widehat{b}_{j,p}'(\eta_m)\d\eta_m, \quad 
[\widetilde{{M}}_{l,m}^{(k)}]_{i,j}:=\int_{0}^1\beta_{l,m}^{(k)}(\eta_m)\widehat{b}_{i,p}(\eta_m)\widehat{b}_{j,p}(\eta_m)\d\eta_m.
\end{equation*}

We also need to define  modified mass matrices  
$$
\widetilde{\mathbf{M}}^{(k)}:=\begin{bmatrix}
\widetilde{\mathbf{M}}^{(k)}_{1} & &\\
& \ddots  \\
& & \widetilde{\mathbf{M}}^{(k)}_{d}
\end{bmatrix}\qquad  \text{for } k=1,\dots,\numpatch,
$$
where, for $l=1,\dots,d$ we defined
$$
\widetilde{\mathbf{M}}_l^{(k)}:=\widetilde{M}_{l,d}^{(k)}\otimes\dots \otimes\widetilde{M}_{l,1}^{(k)}.
$$ 
We then define  for $l=1,\dots,d$ and for $k=1,\dots,\numpatch$ the Schur complement matrices  of $\mathbf{\widetilde{A}}_{l}^{(k)}$, obtained by eliminating the internal degrees-of-freedom,  as \[ \mathbf{\widetilde{S}}_{l}^{(k)}:=\left(\mathbf{\widetilde{A}}_l^{(k)}\right)_{\Gamma \Gamma}-\left(\mathbf{\widetilde{A}}_l^{(k)}\right)_{\Gamma I}\left[\left(\mathbf{\widetilde{A}}_l^{(k)}\right)_{I I}\right]^{-1}\left(\mathbf{\widetilde{A}}_l^{(k)}\right)_{I \Gamma }.\]
Finally, we  define for   $k=1,\dots,\numpatch$
 
\[ 
\widetilde{\mathbf{P}}_A^{\mathrm{inex},(k)}:=  \begin{bmatrix}
D_{A,1}^{\tfrac12}\left( \mathbf{\widetilde{A}}_1^{(k)}+
  {\left(H^{(k)}\right)^{-2}}  \, \mathbf{\widetilde{M}}_1^{(k)}\right)D_{A,1}^{\tfrac12} & &\\
   & \ddots &\\
   &  & D_{A,d}^{\tfrac12}\left(\mathbf{\widetilde{A}}_d^{(k)}+
  {\left(H^{(k)}\right)^{-2}}  \, \mathbf{\widetilde{M}}_d^{(k)}\right)D_{A,d}^{\tfrac12}
\end{bmatrix}\]
and 
\begin{equation}
\label{eq:tildeshur}
\widetilde{\mathbf{S}}^{(k)} :=  \begin{bmatrix}
D_{S,1}^{\tfrac12}\mathbf{\widetilde{S}}_1^{(k)}D_{S,1}^{\tfrac12}  & &\\
   & \ddots &\\
   &  & D_{S,d}^{\tfrac12}\mathbf{\widetilde{S}}_d^{(k)}D_{S,d}^{\tfrac12} 
\end{bmatrix} 
\end{equation}

%
%
%
%
%
%
where $D_{A,l}$ and $D_{S,l}$ are diagonal scaling matrices  defined  for $l=1,\dots,d$  as
$$
[D_{A,l}]_{i,i}:=\frac{\left[\mathbf{ {A}}_l^{(k)}+
  {\left(H^{(k)}\right)^{-2}}  \, \mathbf{{M}}_l^{(k)}\right]_{i,i}}{[\mathbf{\widetilde{A}}_l^{(k)}+
  {\left(H^{(k)}\right)^{-2}}  \, \mathbf{\widetilde{M}}_l^{(k)}]_{i,i}} \quad \text{for } i=1,\dots,\mathcal{N}^{(k)}_{dof}
$$
and 
$$ [D_{S,l}]_{i,i}:=\frac{\left[\left(\mathbf{A}_l^{(k)}\right)_{\Gamma\Gamma}\right]_{i,i}}{\left[\left(\mathbf{\widetilde{A}}_l^{(k)}\right)_{\Gamma\Gamma}\right]_{i,i}} \quad \text{for } i=1,\dots,\mathcal{N}_{\Gamma}^{(k)}.
$$
Our choice for the modified inexact versions of $\mathbf{P}_A$ and $\mathbf{P}_S$ is therefore
\begin{align}
\widetilde{\mathbf{P}}_A^{\mathrm{inex}}& :=\begin{bmatrix}
 \widetilde{\mathbf{P}}_{A}^{\mathrm{inex},(1)}  & & \\
  & \ddots & \\
    & &  \widetilde{\mathbf{{P}}}_{A}^{\mathrm{inex},(\numpatch)} 
 \end{bmatrix}, \label{eq:geo_A} \\ 
 \widetilde{\mathbf{P}}_S^{\mathrm{inex}}& := \begin{bmatrix}
   \mathbf{B}_{\Gamma}^{(1)} \mathbf{\widetilde{S}}^{(1)}\left(\mathbf{B}^{(1)}_{\Gamma}\right)^T & & \\
  & \ddots & \\
    & &    \mathbf{B}_{\Gamma}^{(\numpatch)}\mathbf{\widetilde{S}}^{(\numpatch)}\left( \mathbf{B}_{\Gamma}^{(\numpatch)}\right)^T
 \end{bmatrix},\label{eq:geo_S} 
\end{align}  
%
%
%
%
%
 Due to the partial inclusion of the geometry information,  the scaling with respect to $H^{(k)} $  in  \eqref{eq:geo_A}  and \eqref{eq:geo_S}  is the same as
for the exact preconditioners \eqref{eq:exact_AP}--\eqref{eq:exact_SP}
and differs from \eqref{eq:inexact_AP} and \eqref{eq:inexact_SP}.  
  
\begin{rmk}
 
\label{rmk:schur_prec}
Following \cite{MR1787293}, we recall that another option for   $\mathbf{P}_S$ that is suited for the non-redundant choice of the Lagrange multipliers is  
\begin{align}
\label{eq:S_non_r}  
   \mathbf{P}_{S}^{\mathrm{ex-nr}}& = (\mathbf{B}_{\Gamma}\mathbf{B}_{\Gamma}^T)^{-1}\mathbf{B}_{\Gamma} {\mathbf{S}}\mathbf{B}^T_{\Gamma}(\mathbf{B}_{\Gamma}\mathbf{B}_{\Gamma}^T)^{-1},  
\end{align}  
instead of \eqref{eq:exact_SP}.
The corresponding inexact choice is obtained by replacing $\mathbf{S}$ with either $\widehat{\mathbf{S}}$ or $\widetilde{\mathbf{S}}$. In particular, the geometry inclusion variant takes the form

\begin{align}
\label{eq:inexact_S_non_r}   
   \widetilde{\mathbf{P}}_{S}^{\mathrm{inex-nr}}& = (\mathbf{B}_{\Gamma}\mathbf{B}_{\Gamma}^T)^{-1}\mathbf{B}_{\Gamma} \widetilde{\mathbf{S}}\mathbf{B}^T_{\Gamma}(\mathbf{B}_{\Gamma}\mathbf{B}_{\Gamma}^T)^{-1},  
\end{align}
where, recalling \eqref{eq:tildeshur}, the matrix $\widetilde{\mathbf{S}}$ is defined as
$$
\widetilde{\mathbf{S}} :=\begin{bmatrix}
\widetilde{\mathbf{S}}^{(1)}& & \\
& \ddots & \\
& & \widetilde{\mathbf{S}}^{(\numpatch)}
\end{bmatrix}.
$$ 
Note that the matrix $\mathbf{B}_{\Gamma}\mathbf{B}_{\Gamma}^T$ is a block-diagonal matrix and its inverse can be easily computed.
\end{rmk}


\section{Numerical results}
\label{sec:numerics}

In this section we assess the performance of the preconditioning strategies.
We compare the  exact and inexact  local solvers    introduced in Section~\ref{sec:exact} and Section~\ref{sec:inexact}, respectively, in three dimensional domains. 
We    use the version of the inexact  local solvers   that incorporates some information on the geometry parametrization, that is detailed in Section~\ref{sec:geo_inclusion}. We show the results only for the choices of $\mathbf{P}_S$ 	\eqref{eq:S_non_r} and \eqref{eq:inexact_S_non_r}, for exact and inexact   local solvers,   respectively, because we experimented that   they provide better performances than \eqref{eq:exact_SP} and \eqref{eq:inexact_SP}, respectively.
We report in the tables below the computational time in seconds needed to solve the preconditioned system. For the inexact   local solvers,   the time   includes also the setup time for   the FD method. For the exact   local solvers,   we exclude the time of formation of the mass matrix $\mathbf{M}$. Indeed,  as it is known, the formation of isogeometric matrices is quite expensive unless ad-hoc routines are used (e.g. weighted quadrature  \cite{MR3811621} or low-rank approach \cite{Mantzaflaris2017}) and in this paper we only focus on the solver.
  In the tables, ``EXACT'' refers to the   choice $\mathbf{P}_A= \mathbf{P}_A^{\mathrm{ex}}$   and $\mathbf{P}_S= \mathbf{P}_{S}^{\mathrm{ex-nr}}$ (see \eqref{eq:exact_AP}  and  \eqref{eq:S_non_r}) while ``INEXACT'' refers to the choice $\mathbf{P}_A= \widetilde{\mathbf{P}}_A^{\mathrm{inex}}$   and $\mathbf{P}_S= \widetilde{\mathbf{P}}_{S}^{\mathrm{inex-nr}}$ (see\eqref{eq:geo_A} and \eqref{eq:inexact_S_non_r}).  
In all our tests we set the Lam\'{e} coefficients equal to $\lambda = \frac{0.3}{0.52}$ and $\mu = \frac{1}{2.6}$;   that  choice  models  steel.  

All experiments are performed by Matlab R2017b and using the GeoPDEs toolbox~\cite{vazquez2016new}, on an Intel Core i7-5820K processor, running at 3.30 GHz, with 64 GB of RAM. 
We force a single-core sequential execution in all our tests. 
Our linear solver is the preconditioned MINRES method with the zero vector as initial guess and with tolerance set equal to $10^{-8}$. The generalized eigendecompositions are computed  using \texttt{eig} Matlab function.

  We recall that $p$ stands for the degree of   the splines, $\numpatch$ is for the number of patches, and $n_{el}^{(k)}$ represents  the number of elements in each parametric direction on a given patch $\Omega^{(k)}$.



\begin{figure}[!ht]
\centering
 \subfloat[Parallelepiped with two levels of refinements
 \label{fig:poisson_3d}]{\resizebox {0.64\columnwidth} {!}{
	 \begin{tikzpicture}	 
	 \draw (0,0) -- (1.5,0) -- (1.5,0.5) -- (0,0.5) -- (0,0);
	 \draw (0.5,0) -- (0.5,0.5);
	 \draw (1,0) -- (1,0.5);
	 \draw (0,0.5) -- (0.05,0.65) -- (1.55,0.65) -- (1.5,0.5);
	 \draw[fill] (1.55,0.65) -- (1.55,0.15) -- (1.505,0) --(1.5,0) -- (1.5,0.5) -- (1.55,0.65);
	 \draw (0.5,0.5) -- (0.55,0.65);
	 \draw (1,0.5) -- (1.05,0.65);
	 
	 \draw[line width=0.2] (0.25,0) -- (0.25,0.5);
	 \draw[line width=0.2] (0.75,0) -- (0.75,0.5);
	 \draw[line width=0.2] (1.25,0) -- (1.25,0.5);
	 \draw[line width=0.2] (0.625,0) -- (0.625,0.5);
	 \draw[line width=0.2] (0.875,0) -- (0.875,0.5);
	 \draw[line width=0.2] (1.125,0) -- (1.125,0.5);
	 \draw[line width=0.2] (1.375,0) -- (1.375,0.5);
	 \draw[line width=0.2] (1.0625,0) -- (1.0625,0.5);
	 \draw[line width=0.2] (1.1875,0) -- (1.1875,0.5);
	 \draw[line width=0.2] (1.3125,0) -- (1.3125,0.5);
	 \draw[line width=0.2] (1.4375,0) -- (1.4375,0.5);	 
	 
	 \draw[line width=0.2] (0,0.25) -- (1.5,0.25);
	 \draw[line width=0.2] (0.5,0.375) -- (1.5,0.375);
	 \draw[line width=0.2] (0.5,0.125) -- (1.5,0.125);
	 \draw[line width=0.2] (1,0.0625) -- (1.5,0.0625);
	 \draw[line width=0.2] (1,0.1875) -- (1.5,0.1875);
	 \draw[line width=0.2] (1,0.3125) -- (1.5,0.3125);
	 \draw[line width=0.2] (1,0.4375) -- (1.5,0.4375);
	 
	 \draw[line width=0.2] (0.025,0.575) -- (1.525,0.575);
	 \draw[line width=0.2] (0.52,0.5425) -- (1.52,0.5425);
	 \draw[line width=0.2] (0.53,0.6125) -- (1.53,0.6125);
	 \draw[line width=0.2] (1.025,0.56) -- (1.525,0.56);
	 \draw[line width=0.2] (1.005,0.525) -- (1.505,0.525);
	 \draw[line width=0.2] (1.035,0.63) -- (1.535,0.63);
	 \draw[line width=0.2] (1.025,0.595) -- (1.525,0.595);
	 \draw[line width=0.2] (0.25,0.5) -- (0.3,0.65);
	 \draw[line width=0.2] (0.75,0.5) -- (0.8,0.65);
	 \draw[line width=0.2] (0.625,0.5) -- (0.675,0.65);
	 \draw[line width=0.2] (0.875,0.5) -- (0.925,0.65);
	 \draw[line width=0.2] (1.25,0.5) -- (1.3,0.65);
	 \draw[line width=0.2] (1.125,0.5) -- (1.175,0.65);
	 \draw[line width=0.2] (1.375,0.5) -- (1.425,0.65);
	 \draw[line width=0.2] (1.1875,0.5) -- (1.2375,0.65); 
	 \draw[line width=0.2] (1.06,0.5) -- (1.11,0.65); 
	 \draw[line width=0.2] (1.4375,0.5) -- (1.4875,0.65); 
	 \draw[line width=0.2] (1.31,0.5) -- (1.36,0.65); 
	 	 \end{tikzpicture}
}} \hfill
 \subfloat[Sphere domain
\label{fig:sphere}]{\resizebox {0.35\columnwidth} {!}{
    \includegraphics{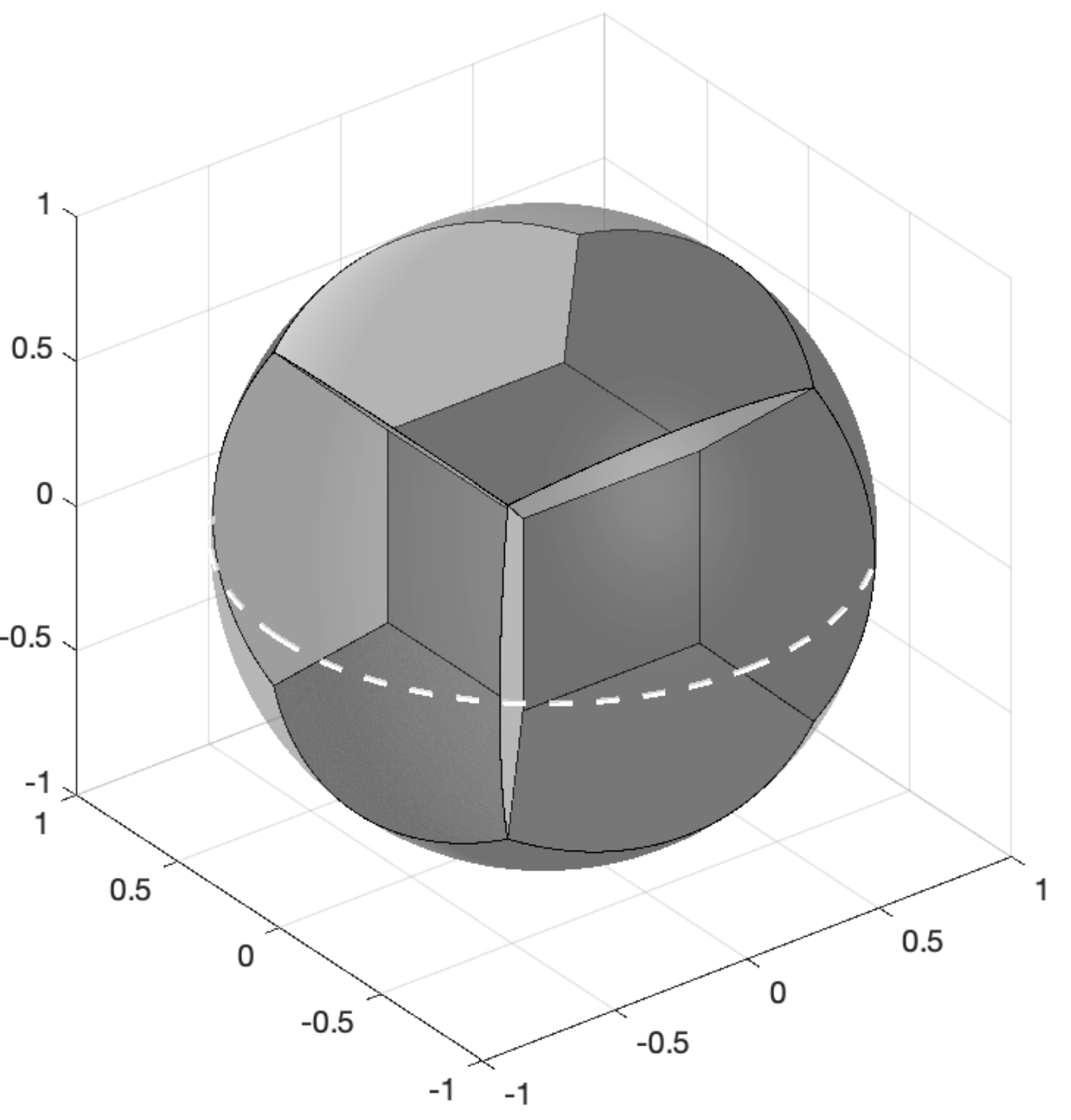}
 }} \hfill

\caption{Computational domains}
\label{fig:non_conforming_domains}
 \end{figure}
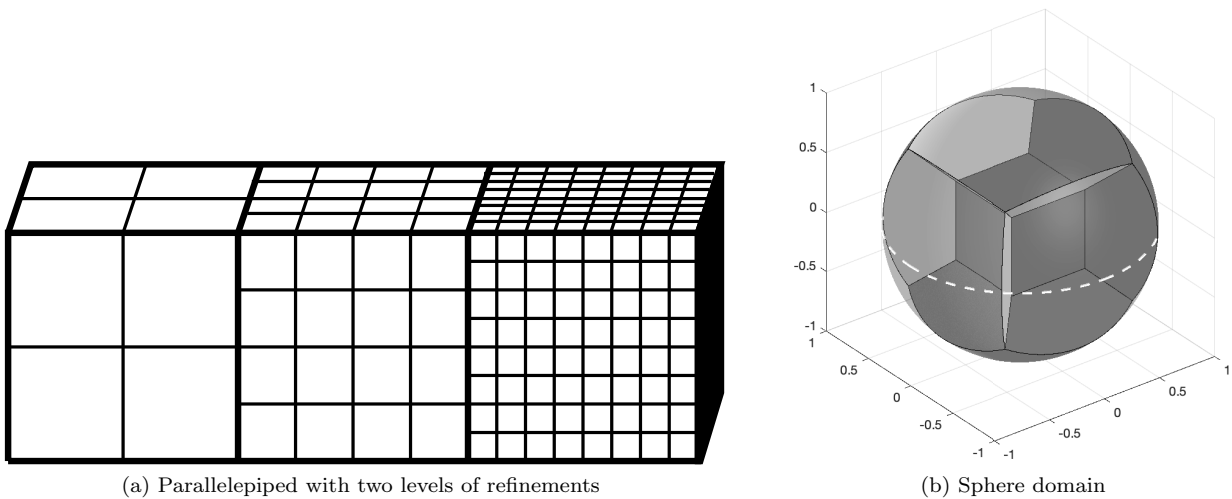
 
\paragraph{Non-conforming test.}  
We focus on the case in which the knot vectors of one patch are a refinement of the knot vectors of the adjacent one: this allows the use of the properties of the knot insertion algorithm to glue the degrees-of-freedom of the interface and properly modify the gluing matrix $\mathbf{B}$  (see~\cite[Section 5.2]{piegl2012nurbs}).   

Let $\Omega = (0,1) \times (0,3) \times (0,1)$ be a domain on which we consider a Dirichlet boundary value problem. We choose the problem data so that the exact solution is $\boldsymbol{u} (x, y, z) = [\cos{x},\ z\sin{y},\ (xyz)^2]^T$. The domain is divided into three patches, with increasing levels of refinement (see Figure~\ref{fig:poisson_3d}).  In particular, we set $n_{el}^{(1)}=2^2$ in the leftmost patch, $n_{el}^{(2)}=2^3$ in the central patch, and $n_{el}^{(3)}=2^4$ in the rightmost one. 
%
%

\begin{table}[!h]
\caption{Preconditioners comparison - non-conforming test on the parallelepiped domain} 
\label{tab:poisson_3d}
\begin{center}
\begin{adjustbox}{max width=\textwidth}
 \begin{tabular}{c |  c c |  c c }
   & \multicolumn{2}{c |}{\textbf{EXACT}} & \multicolumn{2}{c}{\textbf{INEXACT}}  \\
 
$ {p}$  & Iter. &  {Time (s)} &   Iter. &  {Time (s)}   \\
  \hline
		  {1} &  53  &  \z\z119 &  \z86 & \z2   \\
		 {2} &  57  &  \z\z283   & \z96 &  \z3   \\
		 {3} &  63  &  1 080    & 104 &  \z6  \\
		 {4} &  67 &  3 740  & 111  &    12  \\
	\end{tabular}
\end{adjustbox}
\end{center}
\end{table}

In Table~\ref{tab:poisson_3d}  we report the results for the EXACT and INEXACT approaches. The  INEXACT local solvers yield  to a number of iterations that is less than twice the number of iterations of the  EXACT choice. On the contrary, the solving times of the INEXACT local solvers are orders of magnitudes less than the times of the EXACT local solvers. In all cases, the number of iterations is   only mildly dependent on $p$.     
 
\paragraph{Sphere domain test.}  Let $\Omega$ be the unitary sphere with center in the origin that is divided into  7 patches: an inner cube and six patches around it (see Figure \ref{fig:sphere}).
   We remark that the sphere patches are described by NURBS
parametrizations and that  the isogeometric local spaces used for the
discretization are   spaces of  mapped NURBS. However, the INEXACT
local solvers are still built using B-splines and some information  on
the weights of the NURBS basis functions, as well as some information
about the geometry maps, is incorporated into the solvers  as
described in  Section \ref{sec:geo_inclusion}.    
  We choose the problem data so that the exact solution is  $\boldsymbol{u} (x, y, z) = [\cos{x},\ z\sin{y},\ (xyz)^2]^T$. 
 We impose Dirichlet boundary conditions on   $\partial \Omega\cap \{z\leq0\}$, while   Neumann conditions are imposed  on the rest of the boundary, i.e. on $\partial \Omega\cap \{z>0\}$.
Note that in classical IETI setting, this choice of b.c. provides  local problems without a tensor product structure.   
 
  In this test we consider an equal number of element in each parametric direction of each patch, i.e. we set $n_{el}^{(1)}=\dots = n_{el}^{(7)}=:n_{el}$. We report the numerical results in Table \ref{tab:le_sphere}  for an increasing number of $n_{el}$ and for NURBS of degree 4 and 5. The INEXACT approach is orders of magnitude faster that the EXACT one, even though the number of iterations is higher. 
 

\begin{table}[!h]
\caption{Preconditioners comparison -  sphere domain test.} 
\label{tab:le_sphere}
\begin{center}
\begin{adjustbox}{max width=\textwidth}
 \begin{tabular}{c | c c | c c || c c | c c}
 & \multicolumn{4}{c|| }{$p=4$} & \multicolumn{4}{| c}{  $p=5$}\\
 \hline 
   & \multicolumn{2}{c |}{\textbf{EXACT}} & \multicolumn{2}{c||}{\textbf{INEXACT}}  & \multicolumn{2}{| c }{\textbf{EXACT}} & \multicolumn{2}{|c}{\textbf{INEXACT}} \\
  ${n_{el}}$ &   Iter. &  {Time (s)} & Iter. &  {Time (s)} & Iter. &  {Time (s)} & Iter. &  {Time (s)}  \\
  \hline
	 {\z2} & 61  &  \z\z\z\z\z7 &  123 &  \z\z5   & \z65  &  \z\z\z\z\z22  & 146  & \z\z7 \\
	 {\z4} & 73  & \z\z\z\z56  &  146 &  \z\z9   & \z79  & \z\z\z\z160    & 172  &   \z16  \\
	 {\z8} & 87  &  \z2 757  &  184 &   \z26  &  100 &   \z\z5 764  &  225 &  \z56  \\ 
	 {16}  & 97  &    63 699   &  217 &   132   &  105 &  275 500   & 246  &   277    
	\end{tabular}
\end{adjustbox}
\end{center}
\end{table}

%
  
\paragraph{Weak scalability test.} 
 In this test we consider the weak scalability  
of the AF-IETI  preconditioned method: we keep the dimension of each
subproblem fixed and we increase the number of patches. 
In particular, we consider a cube domain $\Omega = (0,1)^3$. We choose
the problem data so that the exact solution is $\boldsymbol{u} (x, y,
z) = [\cos{x},\ z\sin{y},\ (xyz)^2]^T$. We consider     Neumann
boundary conditions on  $\partial \Omega\cap\{x=0\}$ and $\partial
\Omega\cap\{x=1\}$ and Dirichlet boundary conditions on the remaining
faces.   We   split the domain $\Omega$  into an increasing number
$\numpatch$ of equal cubes,   keeping  $p=3$ and   $n_{el}^{(1)}=\dots  = n_{el}^{(\numpatch)}=2^3$.


\begin{table}[!ht]
\caption{Preconditioners comparison - weak scalability test.}
\label{tab:scalability_3d}
\begin{center}
\begin{adjustbox}{max width=\textwidth}
 { \begin{tabular}{c |  c c | c c  }
  & \multicolumn{2}{c |}{\textbf{EXACT}} & \multicolumn{2}{c}{\textbf{INEXACT}} \\
  $\numpatch$  & Iter. &  {Time (s)} & Iter. &  {Time (s)} \\
  \hline
		$ {2^3}$ &  48  &  \z\z\z\z7 & \z90 &  \z2 \\
		$ {3^3}$ & 52 &  \z\z\z57 & \z97 &   \z7 \\
		$ {4^3}$ & 53  &  \z\z275 & 100  &  18 \\
		$ {5^3}$ &  56 &  1 030 & 105 &   36 \\
		$ {6^3}$ &  57  &  3 019& 105 &  63  \\
	\end{tabular}}
\end{adjustbox}
\end{center}
\end{table} 
Table~\ref{tab:scalability_3d} shows that the  number of iterations of
both the EXACT and INEXACT solvers is independent of the number pf
patches, that is, the solvers are scalable. Again, the INEXACT variant
saves orders of magnitudes of computational time.  
   In the scalability test  
the computational  time   reflects the  single-core
sequential solving of  local problems. Clearly, a  parallel
implementation would lead to a decrease of the computational time for
both solvers.  
 

\section{Conclusions}
\label{sec:con}

In this paper, we studied a combination of the FD  solver with a domain decomposition method  of FETI type in a general multi-patch isogeometric setting, where the subdomains of the method coincide with the isogeometric patches.   We focused on the All-Floating  (domain) version of FETI, where both the continuity and Dirichlet boundary conditions are weakly imposed by Lagrange multipliers. 
We built a preconditioner for the resulting saddle-point linear system in which  the  FD method is used in the  inexact solvers   for the local problems.   
We  also proved the spectral estimates that guarantee  the good convergence properties of the  preconditioned iterative solver.
The comparison of the performances of the exact and inexact local solvers in the preconditioner  on three dimensional compressible linear elasticity model problems clearly showed that the inexact choice brings great improvements, in terms of computational time.  Numerical results demonstrated the weak scalability is not affected by introducing the inexact solvers, and that the overall approach is quite robust with respect to the spline degree $p$. The variant of the inexact local solvers that includes also some information on the geometry parametrization performs well also with distorted geometries.
 

We remark that while the present work is focused on IgA, the range of applicability of AF-IETI with FD-based inexact solvers is in fact not limited to IgA problems. 
Indeed, the presented approach is suitable for any discretization where the domain is split into a number of non-overlapping subdomains, on each of which the discretization has a tensor structure, like e.g. the spectral element method.

\section*{Acknowledgements}
The authors were partially supported by the European Research Council through the FP7 Ideas Consolidator Grant HIGEOM n.616563. The  second, third and fourth authors are members of the Gruppo Nazionale Calcolo Scientifico-Istituto Nazionale di Alta Matematica (GNCS-INDAM)  and the second  author  was partially supported by  INDAM-GNCS  ``Finanziamento Giovani Ricercatori 2019-20" for the project ``Efficiente risoluzione dell'equazione di Navier-Stokes in ambito isogeometrico". These supports are gratefully acknowledged.

 \bibliographystyle{plain}
\bibliography{FD_DDM_IGA}

\end{document}